\documentclass[10pt]{amsart}
\usepackage{amsmath, amssymb, amsthm,  epsfig}
\usepackage{amsfonts}
\theoremstyle{plain}
\newtheorem{theorem}{Theorem}[section]
\newtheorem*{thrm*}{Theorem}
\newtheorem{lemma}[theorem]{Lemma}
\newtheorem{proposition}[theorem]{Proposition}
\newtheorem{corollary}[theorem]{Corollary}
\newtheorem{definition}[theorem]{Definition}
\newtheorem{remark}[theorem]{Remark}
\numberwithin{equation}{section}
\newcommand{\R}{\mathbb R}
\newcommand{\N}{\mathbb N}

\newcommand{\sL}{\mathcal L}

\newcommand{\al}{\alpha}

\newcommand{\p}{\partial}  

\newcommand{\Om}{\Omega}

\newcommand{\e}{\epsilon}

\newcommand{\Hn}{\mathbb H^n}
\newcommand{\Hone}{\mathbb H^1}

\newcommand{\loc}{\text{loc}}

\newcommand{\sign}{\text{sign}}

\def \mX {{\mathcal X}}
\def \mS {{\mathcal S}}
\def \mY {{\mathcal Y}}
\def \R {{\mathbb {R}}}
\def \N {{\mathbb N}}

\def \s {{\sigma}}

\def \O {{\Omega}}
\def \Om {{\Omega}}

\def \p {{\partial}}

\def  \e {{\varepsilon}}

\def \tilde {\widetilde}

\def \O {\Omega}

\def \G {\Gamma}
\def \z {\zeta}

\def\p{\partial}

\def \x {{\xi}}
\def \tX {\tilde X}
\def\te {\tilde{e}}
\def\oe{\bar{e}}

\begin{document}

\title[Regularity of non-characteristic minimal graphs in $\Hone$]
 {Regularity of non-characteristic minimal graphs in  the Heisenberg group  $\Hone$}

\begin{abstract}  Minimal surfaces in the sub-Riemannian Heisenberg group can be constructed by means of a Riemannian approximation scheme, as limit of Riemannian minimal surfaces. We study the regularity of Lipschitz, non-characteristic minimal surfaces which arise as such limits. Our main results are a-priori estimates
on the solutions of the approximating Riemannian PDE and the ensuing  $C^{\infty}$ regularity
of the sub-Riemannian minimal surface along its  Legendrian foliation.
\end{abstract}

\author{Luca Capogna}\address{Department of Mathematical Sciences,
University of Arkansas, Fayetteville, AR 72701}\email{lcapogna@uark.edu}
\author{Giovanna Citti}\address{Dipartimento di Matematica, Piazza Porta S. Donato 5,
40126 Bologna, Italy}\email{citti@dm.unibo.it}
\author{Maria Manfredini}
\address{Dipartimento di Matematica, Piazza Porta S. Donato 5,
40126 Bologna, Italy}\email{manfredi@dm.unibo.it}

\keywords{minimal surfaces, sub-Riemannian geometry, viscosity solutions\\
The authors are partially funded by NSF Career grant DMS-0124318
(LC) and by INDAM
(GC)}

%
\maketitle

\section{Introduction}

The first Heisenberg group $\Hone$ is a Lie group with a
3-dimensional Lie algebra $h=V^1\oplus V^2$ such that
$\dim(V^1)=2, \dim(V^2)=1$, $[V^1,V^1]=V^2$ and $[h,V^2]=0$. Let
$\mS,\mX,\mY\in h$ be any  basis such that $[\mS,\mX]=\mY\in V^2$.
By assigning a left-invariant Riemannian metric $g_0$ on the {\it
horizontal}  sub-bundle $H\Hone$ given by the $V^1$ layer, we
obtain a sub-Riemannian space $(\Hone, g_0)$. We choose $\mS,\mX$
such that they are orthonormal with respect to $g_0$. The
corresponding control metric $d_0$ (the {\it Carnot-Caratheodory
metric} \cite{NSW}) is easily shown to be well defined. We extend
$g_0$ to a (left-invariant) Riemannian metric $g_1$ on the full
tangent bundle of $h$ requiring that $V^2$ and $V^1$ are
orthogonal in this extension. The {\it dilated} metrics $g_{\e}$, $\e>0$
are defined so that $\mS,\mX, \e\mY$ are orthonormal.  We define $d_\e$
to be the corresponding distance function.
%
We define   {\it polarized} coordinates 
$(x_1, x_2, x_3)$ in  $\Hone$   by identifying the
triplet with the point $\exp(x_3 \mathcal S)\exp(x_1 \mathcal X + x_2 \mathcal Y)$. The
Baker-Campbell-Hausdorff formula yields $\mathcal S=\p_{3}$,
$\mathcal X=\p_1+x_3\p_2$ and $\mathcal Y=\p_2$. 

 If $M\subset \Hone$ is a $C^1$ surface, then 
$p\in M$ is called {\it characteristic } if both $\mS,\mX$ are tangent to $M$ at $p$.
An intrinsic graph (see \cite{frss:perimeter,frss:carnot}) is a (non-characteristic) graph of the form 
\begin{equation}\label{int-graph-1}M=\{x_3=u(x_1, x_2)|\,
(x_1, x_2)\in\O\subset \R^2\},\end{equation}
where $\Om\subset \R^2$ is an open set.
%
An analogue of the classical implicit function theorem \cite{frss:carnot,CittiManfredini,ascv} shows that any surface $\{f=0\}$
with $\mS f,\mX f\in C(\Hone)$ can be represented as an intrisic graph, in 
a neighborhood of any of its  non-characteristic points.

The flow associated to the   line
bundle of tangent directions which are also horizontal foliate the
complement of the characteristic locus of the surface. This is called {\it Legendrian foliation} in the literature. Note that the  horizontal tangent bundle
of an intrisic graph \eqref{int-graph-1} is spanned by
the single vector field $T=Xu\mS+\mX|_{u((x_1,x_2),x_1,x_2)}$.  We note that the projection of this vector field on $T\Omega$
yields 
the vector field in $\Omega$,
$$X_{1, u} = \p_1 + u \p_2.$$

\bigskip

\paragraph{\bf Minimal surfaces}
Several equivalent notions of {\it horizontal mean
curvature} $H_0$ for a $C^2$ surface  $M\subset \Hone$ (outside
characteristic points) have been given in the literature. To quote
a few:
 $H_0$ can be defined in terms of the first variation of the area functional
 \cite{dgn:minimal, pau:cmc-carnot, chmy:minimal, RR1, Sherbakova, montefalcone,cdpt:survey}, as horizontal divergence of the horizontal unit normal
 or as limit of the mean curvatures $H_\e$ in the Riemannian metrics $g_\e$ as
 $\e\to 0$.
It is also   well known (see for example \cite{chmy:minimal,
 pau:cmc-carnot, CS}) that $H_0$ coincides with the
curvature of the projection of the Legendrian leaves on the
Horizontal plane.

%
A  $C^2$ non characteristic surface $M\subset \Hone$ is called
{\it minimal} if  it satisfies $H_0=0$ identically. In
particular for a $C^2$ intrinsic graph, a direct computation yields that the PDE can be written in terms of the vector $X_{1,u}$
as follows
\begin{equation}\label{eq0bis}
H_0=X_{1, u} \left( \frac{X_{1, u} u}{\sqrt{1+|X_{1, u} u|^2}}\right)=0.
\end{equation}
A deep result of Ambrosio, Serra-Cassano and Vittone \cite{ascv} shows that 
such PDE continues to hold below the $C^2$ threshold in a suitably weak sense.

\bigskip

\paragraph{\bf Generalized solutions and Riemannian approximants} Because
the horizontal mean curvature arises as first variation of the sub-Riemannian perimeter,
minimal surfaces are critical points of the perimeter. As such
these objects can be interpreted in weak sense, far below the
threshold of $C^2$ smoothness (see \cite{ascv},
\cite{gn:isoperimetric}, \cite{pau:minimal},\cite{pau:obstructions},\cite{chmy:minimal},\cite{chy},\cite{Rito}).
 As an example, starting from the family of
shears $x_2-x_1 x_3+g(x_3)$ (see \cite{pau:obstructions}) one can obtain the  intrinsic graph
$x_3=u(x_1, x_2)=\frac {x_2}{x_1-sgn\,  (x_2)}$ defined in $\Omega=\{(x_1,x_2)\in
\R^2: x_1>1\}$ which is minimal in the sense that it is foliated by
horizontal lifts of segments, it is locally Lipschitz (with respect to the Euclidean metric) but clearly not $C^1$ smooth.

In \cite{pau:minimal} and in \cite{chy}, the authors 
 prove existence of (respectively $W^{1,p}$ and Lipschitz) minimal surfaces using
 the Riemannian approximation scheme: as $\e\to 0$
  $$(\Hone, d_\e) \rightarrow (\Hone, d_0),$$
  in the Gromov-Hausdorff  topology (see \cite[Section 2.4]{cdpt:survey} for a detailed description).
  
The approach to existence of solutions  in these papers is  based on  a-priori
estimates for the minimizers of the approximating Riemannian $g_\e$ perimeter functionals \cite{pau:minimal}
and on solutions of  the "{\it Riemannian"} regularized versions of \eqref{eq0bis}  \cite{chy}.
In adapting the approximation scheme to the intrisic graphs setting we note that  
the minimal surface PDE for the metric $g_\e$  corresponding to 
intrinsic graphs  \eqref{int-graph-1}
 \begin{equation}\label{eqebis}L_\e u =\sum_{i=1}^2X_{i ,u}
\left( \frac{X_{i, u} u}{\sqrt{1+|\nabla_\e u|^2}}\right)=0, \text{ in } \Om\subset \R^2\end{equation}
where  
$$X_{2, u} = \e\p_2 \quad \nabla_\e =(X_{1, u}, X_{2, u}),$$
is a  natural elliptic regularization of  the PDE (\ref{eq0bis}).    As such, it is more amenable to establishing a-priori higher regularity estimates. The difficulty of course
resides in obtaining estimates which are uniform in the parameter $\e$ as $\e\to 0$.


These observations lead us to the definition of the class of minimal surfaces we want to investigate
 \begin{definition}\label{visc}
 We say that a function $u \in Lip(\Omega)$ is a vanishing viscosity
 solution of the equation (\ref{eq0bis}) if there exists
  a sequence of positive numbers $\e_{j}$ with
  $\e_{j} \rightarrow 0$ when $j \rightarrow \infty$, and a
  sequence $(u_{j})$ in $C^\infty(\Omega)$ such that:

(i) $\sum_{i=1}^2 X_{i, u_j} \left(
\frac{X_{i, u_j} u_j} { \sqrt{1+|\nabla_{\e_j} u_j|^2} }\right)=0.
\quad in
 \;\;\Omega\;\;\text{for all } j\in N. $ 
 
 (ii) The sequence $(u_{j})$ is bounded in
 $Lip(\Omega)$ and uniformly  convergent on subcompacts of $\Omega$ to $u$. \end{definition}

\begin{remark} Existence of vanishing viscosity solutions in the case of $t-$graphs, i.e.
graphs of the form $x_{2}=g(s,x_1)$, has been proved in
\cite[Theorem A and Theorem 4.5]{chy}.  In the same paper the authors establish that such solutions are
perimeter minimizers  and address  uniqueness.  Assuming $C^1$ convergence
of the approximating solutions, outside the characteristic sets
the $t-$graphs solutions in \cite{chy}, both the approximating and the limit
solutions,  can be represented as intrisic graphs and hence yield vanishing viscosity
intrinsic graphs.
\end{remark}
%
%

\paragraph{\bf Regularity results}
Given the examples of non-smooth minimal surfaces mentioned above,  the question
arises as which kind of regularity can one expect. This problem has beeen recently addressed in a series of papers \cite{pau:obstructions}, \cite{chy}, \cite{CCM2}, \cite{chy:regular} and \cite{fsc-bigolin}. Valuable insights into the problem of regularity
are also provided in the 
works \cite{pau:obstructions}, \cite{chy}, and \cite{Rito} in the form of examples of non-smooth minimal surfaces.
 The regularity properties of the implicit function in the implicit
function theorem quoted earlier provides an interesting insight
into this problem and indicates that one should look for
regularity only in the direction of the Legendrian foliation. 
Indeed we prove

 \begin{theorem}[Regularity]\label{main}If $u$ is
 a vanishing viscosity solution of \eqref{eq0bis},
 then for all $\al\in (0,1)$ and $K\subset \subset \Om\subset \R^2$,
  \begin{equation}\label{c1alpha}
  u\in C^{1,\alpha}(K)
  \end{equation} and for all $k\in \N$ and $p>1$
  \begin{equation}\label{higher}X^k_{1, u} u \in W^{1,p}_{loc} (\Omega).
  \end{equation}  Here
  $C^{1,\alpha}(K)$ and  $ W^{1,p}_{loc} (\Omega)$ denote the
spaces of functions with H\"older continuous Euclidean gradient
 and the classical Sobolev Space. In particular  one has
 $X^k_{1, u} u\in C^{\al}_{loc}(\Om)$ (the Euclidean H\"older space) for all
 $\al\in (0,1)$
 and  hence $X^2_{1, u} u=0$,  holds pointwise everywhere.
 \end{theorem}

\begin{remark}
To better understand the notion of  intrinsic regularity  we return to
the  non-smooth minimal graph $u(x_1, x_2)=\frac {x_2}{x_1-sgn\,  (x_2)}$
described earlier. Although this function is not $C^1$  in the
Euclidean sense, observe that $X_1 u=0$ for every $x_1, x_2\in\Omega$
Hence, this is an example of a minimal surface which is not smooth
but which can be differentiated indefinitely in the direction of the Legendrian foliation.
\end{remark}

\begin{remark}
 The regularity theory for intrinsic minimal surfaces in $\Hn$ with
$n>1$ is quite different. In the recent paper \cite{CCM2} we
show that any Lipschitz continuous  vanishing viscosity minimal
intrinsic graph in $\Hn$, $n>1$ (defined through the Riemannian
approximation scheme) is smooth. The main reason is that in higher
dimension the horizontal tangent bundle generates as a Lie algebra
the full tangent bundle, while this does not happen in the $n=1$
case.
\end{remark}

As a consequence of the regularity theorem we can prove that the
Sobolev weak derivatives of vanishing viscosity solutions agree
with Lie derivatives along the leaves of the Legendrian foliation.
Hence, we obtain that vanishing viscosity solutions  actually satisfy \eqref{eq0bis}
 everywhere  pointwise. This result immediately yields a {\it rigidity} of the Legendrian 
 foliation.
\begin{corollary}[Lie differentiability and Legendrian foliation]\label{teo01} Let $x_3 = u(x_1, x_2)$, $(x_1, x_2)\in \O$ be a Lipschitz continuous vanishing viscosity minimal graph.
 The flow of the vector $X_{1,u}$ yields a foliation of the domain $\O$  by polynomial curves $\gamma$ of degree two. For every fixed $x_0 \in \O$
denote by $\gamma$ the   unique leaf  passing through that fixed
point. The function $u$ is differentiable at $x_0$ in the Lie sense along
$\gamma$ and the equation \eqref{eq0bis} reduces to
$$\frac{d^2}{dt^2}(u (\gamma(t)))=0.$$
\end{corollary}

\paragraph{\bf Comparison with other regularity results}
We describe the relation between our results
in Theorem \ref{main} and Corollary  \ref{teo01} and
the regularity results in \cite{chy:regular} and in \cite{fsc-bigolin}. In \cite{chy:regular} Cheng, Hwang and Yang prove that any $C^1$ weak solution of the prescribed (continuous) horizontal mean curvature PDE, has $C^2$ smooth Legendrian foliation outside of the characteristic set.
 In \cite{fsc-bigolin}, Bigolin and Serra Cassano study the 
regularity of minimal intrisic graphs  \eqref{int-graph-1}
where  \eqref{eq0bis} is interpreted in a weak sense  (i.e. $broad^*$ solutions
defined in  \cite[Definition 3.1]{fsc-bigolin}) and prove (among other results) that Lipschitz regularity
of the intrisic gradient $X_{1,u}u$ implies the Euclidean Lipschitz regularity of the function $u$. 
In the present  paper we require only Lipschitz continuity  of $u$ and prove higher order
 intrinsic differentiability than either \cite{fsc-bigolin}, and \cite{chy:regular}.  On the other hand,
 we  only deal with the case $H_0=0$ and with those  solutions which are  limits of Riemannian minimal graphs. In this sense our results are more specialized than the ones in the
 other two papers.
 
\paragraph{\bf Applications of Theorem \ref{main}} 
Invoking the implicit function theorem, we 
can apply Theorem \ref{main} to study  the
regularity   away from the characteristic locus of the Lipschitz
perimeter minimizers found in \cite{chy}. Since the results in that paper apply
to $t-graph$ and not intrinsic graphs we need some extra assumptions on the convergence of the approximating solution to be able to invoke our intrinsic graphs regularity.
Here and in the following
$\nabla_E$ denotes the Euclidean gradient in $\R^{2}$,
and $(z_1,z_2,z_3)$ are the exponential coordinates
$\exp(z_1 \mS+z_2 \mX+z_3 \mY)=\exp(x_3 \mathcal S)\exp(x_1 \mathcal X + x_2 \mathcal Y)$.

\begin{corollary}\label{chy}
Let $O\subset \R^{2}$ be a strictly  convex, smooth open set,
$\phi\in C^{2,\alpha}(\bar O)$ and
for each $(z_1,z_2)\in O$ denote by $(z_1,z_2)^*=(z_2,-z_1)$. 
Consider the family  $$\{g_\e(z_1,z_2)\}_\e \ \ 
\sup_O |g_\e|+\sup_O|\nabla_E g_\e| \le C \ \ \  \text{(uniformly in }\e),$$
of   smooth  solutions of 
the approximating minimal surface PDE
$$div\Bigg(\frac{\nabla_E g_\e + (z_1,z_2)^*}{\sqrt{\e^2+|\nabla_E g_\e + (z_1,z_2)^*|}}\Bigg)=0 \text{ in }O \ \text{ and }g_\e=\phi \ \text{ in }\p O$$
found in \cite[Theorem 4.5]{chy}.  If for $p_0=(p_0^1,p_0^2)\in O$, $a>0$ and for every $\e>0$ we have $|\p_{z_1} g_\e(p_0)|>a>0$  (or any other partial derivative 
is non-vanishing at $p_0$ uniformly in $\e$) then 
there is a sequence $\e_k\to 0$ such that the Lipschitz perimeter minimizer $g=\lim_{\e_k\to 0}g_{\e_k}$ satisfies 
$g\in C^{1,\alpha} $ and is infinitely many times differentiable in the direction
of the Legendrian foliation of $z_3=g(z_1,z_2)$, in a neighborhood
of the point $p_0$.
\end{corollary}

\begin{proof}
The implicit function theorem implies that the level set of $$z_{3}-g_\e(z_1,z_2)$$ can be written as  smooth intrinsic graphs $x_3=u_\e(x_1,x_2)$ in a neighborhood $\Omega$ of
$(p_0^1,p_0^{2},g(p_0))$. The Lipschitz bounds on $g_\e$ (proved in \cite[Propositions 4.2-4]{chy})
yield uniform Lipschitz bounds on $u_\e$, thus allowing to apply
Theorem \ref{main} and conclude the proof.
\end{proof}

From this result one may conclude immediately the following

\begin{corollary}
 Let $z_3 = g(z_1, z_2)$, $(z_1, z_2)\in \O\subset \R^2$ be a $C^1$  minimal graph
 which is the $C^1$ limit of Riemannian minimal graphs as in  \cite[Theorem 4.5]{chy}. In the 
 neighborhood of any non-characteristic point $g\in C^{1,\alpha} $ and is infinitely many times differentiable along the Legendrian foliation.
\end{corollary}

 Theorem \ref{main} can be also used to rule out minimal intrinsic graphs which do not arise
 as limits of Riemannian minimal graphs. For instance, the example exhibited above
 $x_3=u(x_1, x_2)=\frac {x_2}{x_1-sgn\,  (x_2)}$ defined in $\Omega=\{(x_1,x_2)\in
\R^2: x_1>1\}$ 
 cannot be a vanishing viscosity minimal graph as it lacks the $C^{1,\alpha}$ regularity
 from Theorem \ref{main}.

\paragraph{\bf Sketch of the proof and final  remarks}
We now turn to a description of the techniques used in the proof of Theorem \ref{main}. Since $u$ is a vanishing viscosity solution, this theorem is proved
by means of a priori estimates (uniform
in the parameter $\e$ as it decreases to zero) for each element of the
approximating sequence $u_j$ of solutions of  \eqref{eqebis}.

 The PDE  \eqref{eq0bis}, has a structure similar to the Levi equation in $\R^3$.
 In fact,  the Levi equation can be represented as Riemannian approximation of a sum of squares of vector fields
 $$Z^2_u + W_uu^2 + \e^2 \p_{2}^2u=0,$$
 for suitable non linear vector fields $Z_u,$ $W_u$ 
 depending on the solution $u$. 
 Regularity results for Lipschitz continuous viscosity solutions were established  in \cite{CLM} and 
 \cite{CM}. The techniques used in these papers are based on 
  a
 modification of the Moser iteration, along with representation formula and
 uniform estimates on the fundamental solution. In \cite{CM},
 \cite{CittiTomassini}  the authors address properties of the analogue of the Legendrian foliation  for a  Levi flat graph. 
 
 The cited work  on the Levi equations provides a coarse outline and a strategy for the  proof of Theorem \ref{main}. 
However,  
the equation  \eqref{eq0bis}  presents  additional difficulties (lack of a background sub-Laplacian, worse nonlinearity) with respect to  the Levi equation, and the 
adaptation of the known techniques is very non trivial and requires new ideas.

The first step in the proof involves the  study of a {\it linearization} (of sorts) of  \eqref{eqebis} (for simplicity we will refer
to the approximating functions $u_j$ simply as $u$)

 \begin{equation}\label{eqebi1s} M_{\e,u} z= \sum_{i=1}^2 X_{i ,u}\left( \frac{X_{i, u} z}{\sqrt{1+|\nabla_\e u|^2}}\right)=0, \text{ in } \Om\subset \R^2.
\end{equation}

There are two main difficulties in establishing a-priori estimates, uniform in $\e$, for 
this PDE: 

\begin{itemize}
\item The first problem is due to the fact that  the coefficients of the equation involve the function $u$ which although smooth satisfies a-priori bounds which are uniform  in $\e$  only for the Lipschitz norm. To deal with this obstacle we operate a
{\it freezing} argument, substituting the function $u$ with an analogue of
its first order Taylor polynomial, and then carefully study the remainder terms.
The regularity of rough coefficients degenerate elliptic PDE has been studied by many authors, see for instance the monograph \cite{sawhe} for a survey of the literature and
new, ground-breaking techniques.

\item The second difficulty stems from the fact that, although \eqref{eqebi1s} is elliptic, its coerciveness
degenerates as $\e\to 0$. Now, the approximation of degenerate elliptic operators
with elliptic {\it regularization} is a well known and widely used trick. For instance
in \cite{hormander}, the sub-Laplacian $\sL u= \sum_{i=1}^m X_i^2u$ associated to 
a system of H\"ormander vector fields is approximated by $\sL_{\e} u= \sL u+ \e \Delta u$. While it is true that the ellipticity of $\sL_\e$ degenerates as $\e\to 0$, and hence
the constants involved in elliptic estimates degenerate as well, the operator
$\sL_\e$ satisfies {\it sub-elliptic }estimates, which are uniform in $\e$.
In our case however, the left-hand side of equation  \eqref{eqebi1s} approximates not a sub-Laplacian
but the operator  $X_{1,u}(X_{1,u} z/\sqrt{1+(X_{1,u}u)^2})$ which is not sub-elliptic.
To solve this problem, and obtain  the regularity in $L^p$ of the derivatives of $z$,
  we introduce a completely new ad-hoc  lifting process,  inspired in part by Rothschild and Stein's techniques in \cite{Roth:Stein}. The vector fields $X_{1,u}, X_{2,u}$
  are lifted to a three-dimensional space $\Omega\times(-1,1)$ by adding
  a new variable $s$ and horizontal vector field $\p_s$. The  lifted vectors are
$\tX_1=\p_{x_1}+(u(x)+s^2)\p_{x_2}, \ \tX_2=\e\p_{x_2},  \text{ and }\tX_3=\p_s$
and the set $\{\tX_1,\tX_3\}$
 form a step-three bracket generating system whose commutators yield the direction of degeneracy $\p_{x_2}$ of  \eqref{eqebis}.
 At this point we 
  operate a freezing argument and  approximate
the   operator $M_{\e,u}$  with  higher dimensional H\"ormander type sub-Laplacians (namely
\eqref{nondivfroz} and \eqref{nondivfroz0}, 
built from the frozen, lifted vector fields.
In this way we extablish a priori $W^{2,p}$ estimates  uniform in $\e$ (Theorem
\ref{p402}), which will be the starting point 
of the regularity proof. 
\end{itemize}

In the proof of Theorem \ref{main} we will switch back and forth from representations
of \eqref{eqebis} (and its differentiated versions) in divergence and in non-divergence
form. The former works best to deal with higher regularity, via Caccioppoli estimates.
The second is tailor-made for the freezing technique and the $W^{2,p}$ estimates.

An important ingredient in the proof is the recent result in \cite{CiMa-F}, establishing
uniform estimates on the fundamental solutions of Riemannian regularizations of sub-Laplacians
(see the statement in Theorem \ref{fundam}).

\bigskip

%



\bigskip

For other aspects of minimal surfaces in the Heisenberg group,
including classification and Bernstein-type results, see
\cite{ch:minimal,gp:bernstein,bascv, dgn2,dgnp1,RR,RR1,Rito}.
 These works also contain more comprehensive lists of references.

\bigskip
\paragraph{\it Acknowledgements} The authors would like to thank
Manuel Ritor\'e,  Francesco Serra Cassano and Paul Yang  for sharing with them their preprints \cite{Rito}, \cite{chy:regular} and \cite{fsc-bigolin}.


\section{Preliminaries}

In this section we will always assume that $u$ and $f$ are  fixed smooth functions
defined on an open set $\O$ of $\R^2$,and that $u$ is a  solution of the PDE
$L_\e u = f$ in $\Om$. In particular we remark that
\begin{equation} \label{MMMM= }||u||_{L^\infty(\Omega)} + ||\nabla_{\e}u||_{L^\infty(\Omega)}+||\p_2 u||_{L^\infty(\Omega)}<\infty,
\end{equation}
and we set
\begin{equation} \label{MM}
M= ||u||_{L^\infty(\Omega)} + ||\nabla_{\e}u||_{L^\infty(\Omega)}+||\p_2 u||_{L^\infty(\Omega)}
\end{equation}
where 
 for any function $\phi$
defined on $\O$ we have let  $\nabla_\e \phi=(X_{1,u} \phi, X_{2,u}\phi)$.

We will use the notation $W^{1,p}_{\e}(\O)$, $p>1$ to denote the
Sobolev space corresponding to the norm
$||\phi||_{W^{1,p}_{\e}(\O)} =||\phi||_{L^p(\O)}+||\nabla_\e
\phi||_{L^p(\O)}$. For simplicity, unless we want to stress the
dependence on $u$, we will simply write $X_1,X_2$ instead of $X_{1, u}, X_{2,u}$.
We will denote by $W^{k,p}_0(\O)$ the space of $L^p(\Om)$ functions $\phi$ such that
$X_1^l \phi\in L^p(\Om)$ for all $1\le l \le k$.
\bigskip

We recall that, under assumption (\ref{MMMM= }), the following
result holds, (see \cite{CCM2})

\begin{proposition}\label{Calphaestimate}
Let  $u$ be a solution of equation (\ref{eqebis}) satisfying (\ref{MMMM= }). For
every compact set $K\subset\subset \Omega$ then there exist a real
number $\alpha$ and a constant $C$, only dependent on the bounds
on the constant $M$ in (\ref{MM}) and on the choice of the compact
set $K$ such that

$$||u||_{W^{2,2}_\e(K)} + ||\p_{2}u||_{W^{1,2}_\e(K)} + ||u||_{C^{1,\alpha}_u(K)}\leq C.$$
\end{proposition}

\subsection{An interpolation inequality}

  \begin{proposition}\label{3.3} For every $p \geq 3$,
there exists a constant $C $, dependent on $p$, and the
constant $M$ in (\ref{MM}) such that for every function $z\in
C^{\infty}(\O)$ and for every  $\phi \in C_0^\infty(\O)$, and
every $\delta>0$ 
$$ \int |X_{i}z| ^{p+1/2} \phi^{2p} \leq
\frac{C}{\delta} \int \Big(|z|^{4p+2} \phi^{2p} + |z|^{(2p+1)/2} |X_i\phi|^{(2p+1)/2} +  |z|^{(2p+1)/2} \phi^{2p} \Big) $$$$ + \delta \int
|\nabla_\e(|X_{i} z|^{(p-1)/2})|^{2} \phi^{2p} ,$$ 
where $i$
can be either $1$ or $2$.
\end{proposition}

\begin{proof}
This is a slight variant of  \cite[Prop. 4.2]{CLM}.
We have $$ \int |X_iz| ^{p+1/2} \phi^{2p} = \int X_iz |X_iz|
^{p-1/2} \sign(X_iz)\phi^{2p}  =$$ (integrating by parts and using the\,
fact that $X_1^* = - X_1 - \p_2u  $ and $X_2^*=-X_2$) $$= -\int \delta_{i1} \p_2u \,z
|X_iz| ^{p-1/2} \sign(X_iz)\phi^{2p} -(p-1/2) \int z X_i^2z|X_iz|^{p-3/2}
 \phi^{2p}$$
\begin{equation}\label{eq37}
 - 2p \int z |X_iz|^{p-1/2} \sign(X_iz)\phi^{2p-1} X_i\phi \leq
\end{equation} (where $\delta_{ij} $ denotes the Kroeneker's delta, by H\"older inequality)
 $$\leq \frac{C}{\delta} \int z^{(2p+1)/2} (\phi^{2p} + |X_i\phi|^{(2p+1)/2}) +  
 \frac{C}{\delta} \int z^{4p+2} \phi^{2p}   +  $$$$ +
 \delta  \int |X_iz| ^{p+1/2} \phi^{2p} 
+ \delta \int
|\nabla_\e(|X_{i} z|^{(p-1)/2})|^{2} \phi^{2p} ,$$
 choosing $\delta>0$
sufficiently small we conclude the proof.
\end{proof}

\bigskip

A slight modification of the previous proposition, 
is the following: 
\begin{proposition}\label{interpinfinity}
  For every $p \geq 3$, for every function $z\in
C^{\infty}(\O)$ 
there exists a constant $C$, dependent on $p$, the
constant $M$ in (\ref{MM})  such that and for every  $\phi \in C_0^\infty(\O)$, and
every $\delta>0$
 $$ \int |X_{i}z| ^{p+1} \phi^{2p} \leq$$$$\leq C \left(\int \Big(z^{p+1} \phi^{2p} + z^2 |X_iz|^{p-1}\phi^{2p-2}|X_i\phi|^{2}\Big) +  
   \int |X_i^2 z|^2|X_iz|^{p-3}|z|^2\phi^{2p}\right),$$  where $i$
can be either $1$ or $2$.
\end{proposition}

\begin{proof} We have $$ \int |X_iz| ^{p+1} \phi^{2p} = \int X_iz |X_iz|
^{p} \sign(X_iz)\phi^{2p}  =$$ (integrating by parts and using the
fact that $X_1^* = - X_1 - \p_2u  $ and $X_2^*=-X_2$) $$= -\int \delta_{i1}\p_2u\, z\,
|X_iz| ^{p} \sign(X_iz)\phi^{2p} -p \int z X_i^2z|X_iz|^{p-1}
 \phi^{2p}$$
\begin{equation}\label{eq37}
 -2p \int z |X_iz|^{p} \sign(X_iz)\phi^{2p-1} X_i\phi \leq
\end{equation} (by H\"older inequality)
 $$\leq \frac{C}{\delta} \int \Big(z^{p+1} \phi^{2p} + z^2 |Xz |^{p-1}\phi^{2p-2}|X_i\phi|^{2}\Big) +  
 \delta  \int |X_iz| ^{p+1} \phi^{2p} 
+ \frac{C}{\delta} \int |z|^2|X_i^2 z|^2|X_iz|^{p-3}\phi^{2p},$$ choosing $\delta$
sufficiently small we obtain the desired inequality.
\end{proof}

%

\subsection{The  horizontal mean curvature  as a divergence form operator}

We now  prove  that if $u$ is a smooth solution of equation (\ref{eqebis})  
then its   derivatives $\p_{2} u$ and $X_{k} u$ are solution of a
similar mean curvature equation with different right hand side (see
also \cite[Lemma 3.1]{CCM2}).
Differentiating the equation $L_{\e}u=0$  with respect to $X_{k}$
one obtains

\begin{lemma}\label{equXz} If $u$ is a smooth solution of $L_{\e}u=0$ then $z=X_{k} u$
with $k\leq 2$ is a solution of the equation
\begin{multline}\label{PDEXu}
X_{i}  \Bigg( \frac{a_{ij}(\nabla_\e u)}{\sqrt{1 +|\nabla_\e u|^2}}
X_{j}z \Bigg) \\ =-[X_{k}, X_{i}] \Bigg(\frac{X_{i}u}{\sqrt{1 +
|\nabla_{\e} u |^2}}\Bigg) \\ - X_{i} \Bigg(\frac{a_{ij}(\nabla_\e
u)}{\sqrt{1 +|\nabla_{\e} u|^2}} [X_{k}, X_{j}] u\Bigg),
\end{multline}
where $a_{ij}$ are defined as 
\begin{equation}\label{aij} a_{ij}:R^{2}\rightarrow
R\quad  a_{ij}(p)= \delta_{ij} - \frac{p_ip_j}{1 + |p|^2}.\end{equation}
\end{lemma}

\begin{lemma}\label{equ2n}
If $u$ is a smooth solution of $L_{\e}u=0$ then $v=\p_{2} u $ is a
solution of the equation
\begin{multline}\label{2nderiv} 
X_{i}  \Big( \frac{a_{ij}(\nabla_\e u)}{\sqrt{1 +|\nabla_\e u|^2}}
X_{j}v \Big) \\= -\frac{a_{11}(\nabla_\e
u)}{\sqrt{1 +|\nabla_{\e} u|^2}} v^3 - \frac{a_{1j}(\nabla_\e
u)}{\sqrt{1 +|\nabla_{\e} u|^2}}v X_j v \\- X_{i} \Big(\frac{a_{i1}(\nabla_\e
u)}{\sqrt{1 +|\nabla_{\e} u|^2}}  v^2\Big)
\end{multline}
where $a_{ij}$ are defined in (\ref{aij}).
\end{lemma}

\begin{proof}  Differentiating the PDE we obtain
\begin{eqnarray}X_{i}  \Big( \frac{a_{ij}(\nabla_\e u)}{\sqrt{1 +|\nabla_\e u|^2}}
X_{j}v \Big) \!\!\!\!&&=
-[\p_2, X_{i}] \Big(\frac{X_{i}u}{\sqrt{1 +
|\nabla_{\e} u |^2}}\Big) - X_{i} \Big(\frac{a_{ij}(\nabla_\e
u)}{\sqrt{1 +|\nabla_{\e} u|^2}} [\p_2, X_{j}] u\Big) \notag \\
 \!\!\!\! &&=
-[\p_2, X_{1}] \Big(\frac{X_{1}u}{\sqrt{1 +
|\nabla_{\e} u |^2}}\Big) - X_{i} \Big(\frac{a_{i1}(\nabla_\e
u)}{\sqrt{1 +|\nabla_{\e} u|^2}} [\p_2, X_{1}] u\Big)\notag  \\ 
\!\!\!\! &&=-v\p_2 \Big(\frac{X_{1}u}{\sqrt{1 +
|\nabla_{\e} u |^2}}\Big) - X_{i} \Big(\frac{a_{i1}(\nabla_\e
u)}{\sqrt{1 +|\nabla_{\e} u|^2}}  v^2\Big). \notag
\end{eqnarray}
\end{proof}

\bigskip
Let us consider the linear equation satisfied by the components of
the gradient of $u$:

$$M_{\e} z = X_{i}\Big(\frac{a_{ij}(\nabla_\e u)}{\sqrt{1 + |\nabla_\e
u|^2}}X_{j}z \Big).$$
If $z$ is a smooth solution of equation
\begin{equation}\label{eq111}M_\e z = f,
\end{equation}
 then its intrinsic derivatives $X_{i}z$   are still
solutions of the the same equation with a different right-hand
side.

\begin{lemma}\label{lem4}
If $z$ is a smooth solution of (\ref{eq111}) then $s_{1} = X_{1}z
$ is a solution of the equation
\begin{equation}\label{eq30}
 M_\e s_1 = X_{1}f +
X_{i}\Bigg(\frac{a_{i2}(\nabla_\e u)}{\sqrt{1 + |\nabla_\e u|^2}}
\p_2 u X_{2}z \Bigg) -\end{equation}$$-
X_{i}\Bigg(X_{1}\Big(\frac{a_{ij}(\nabla_\e u)}{\sqrt{1 +
|\nabla_\e u|^2}}\Big) X_{j}z \Bigg)
 + \p_2 u  X_{2}\Bigg(\frac{a_{2j}(\nabla_\e u)}{\sqrt{1 +
|\nabla_\e u|^2}}X_{j}z \Bigg).$$
\end{lemma}

\bigskip

An analogous computation ensures that

\begin{lemma}\label{lem6}
If $z$ is a solution of (\ref{eq111}) then $s_{2} = X_2z $ is a
solution of the equation
\begin{equation}\label{eq340} M_\e s_2= X_2f -
X_{i}\Bigg(\frac{a_{i1}(\nabla_\e u)}{\sqrt{1 + |\nabla_\e u|^2}}
\p_2 u X_2z \Bigg) -\end{equation}$$-
X_{i}\Bigg(X_2\Big(\frac{a_{ij}(\nabla_\e u)}{\sqrt{1 + |\nabla_\e
u|^2}}\Big) X_jz \Bigg)
 - \p_2 u  X_2\Bigg(\frac{a_{1j}(\nabla_\e u)}{\sqrt{1 +
|\nabla_\e u|^2}}X_jz \Bigg).$$
\end{lemma}

\bigskip

\section{The horizontal mean curvature as a  non-divergence form operator: $W^{2,p}_\loc$ a priori  estimates.}

The operator $L_\e$ defined in (\ref{eqebis}) can be represented in non-divergence form
\begin{equation}\label{nondiv}N_\e u= \sum_{i,j=1}^{2} a_{ij}(\nabla_\e u)
X_{i}X_{j}u,\end{equation} where 
$ a_{ij}$ are defined in (\ref{aij}). 

Following the approach in the papers \cite{CLM,CM}  we {\it linearize} the operator
$N_\e$ in the following way: While the coefficients of the vector fields 
$X_{i}$ depend on a fixed function $u$,  they will be  applied to 
an arbitrary function $z$, sufficiently regular. The associated linear non divergence form  operator is
\begin{equation}\label{nondivlin}N_{\e, u}z=
\sum_{i,j=1}^{2} a_{i j}(\nabla_\e u) X_{i, u}X_{j,
u}z,\end{equation} where the coefficients $a_{i j}$ are defined in (\ref{aij}).

The main result of this section is the following

\begin{theorem}\label{p402}
Let us assume that  $z$ is a classical solution of $N_{\e, u}z = 0.$ 

\medskip

\noindent {\it (i)}  Let us assume that 
$\alpha\in]0,1[,$ $p>10/3$ and for every
$K\subset\subset \Omega$ there exists a constant $C$ such that
$$||u||_{C^{1,\alpha}(K)} + ||\p_{2} z||_{L^p(K)} + ||\p_{2} X_{u}z||_{L^2(K)} + ||\nabla_{\e}^2 z||_{L^2(K)} \leq C.$$
  Then for any compact set $K_1 \subset \subset K$, there exists a
  constant $C_1$ only dependent on $K$, C, and on the constant in \ref{MM} such that
   $$||z||_{W^{2,10/3}_\e(K_1)}\leq C_1.$$

\medskip

\noindent {\it (ii)}  If, in addition to the previous conditions, there exists a constant $\tilde C$ such that 
   $$ ||\p_{2} X_{u}z||_{L^4(K)} \leq \tilde C,$$
 with $\alpha \geq 1/4$, then for every $p>1$ there exists a constant $C_1$ only dependent on $C$ and $\tilde C$ and $p$ such that
     $$||z||_{W^{2,p}_\e(K_1)}\leq C_1.$$
\end{theorem}


\subsection{Lifting and freezing} The operator $N_{\e}$ is an elliptic (Riemannian) approximation of the sub-Riemannian mean  curvature 
operator in the right-hand side of \eqref{eq0bis}. Its linearization $N_{\e, u}$ can be interpreted an uniformly 
elliptic operator, with least eigenvalue depending on $\e$. 
It is well known that this approximating operator 
has a fundamental solution, but its estimates strongly depend on $\e$. 
In order to obtain estimates uniform in $\e$ 
we further approximate it with an H\"ormander type operator, a sum of squares of 
vector fields,
which has a similar behavior in the direction $X_1$, but for which the 
direction $\p_2$ is the direction of one of the commutators  ( a step-three commutator!). 
The idea is to use a  new version of the famous Rothschild and Stein  lifting theorem, 
only partially inspired to the  procedure in \cite{Roth:Stein}.  
In order to deal with the non-smoothness of $u$, we will also operate a freezing:
roughly speaking we approximate the coefficients of the vector field $X_1$ with their first order 
Taylor polynomials.

\medskip

\paragraph{\bf Lifting.} The first step is to lift the vector fields to a higher dimensional space through the
introduction of a new variable  $s$. The points of the extension space will be
denoted $(x,s)\in \Omega\times(-1,1)\subset R^3$, with $x=(x_1,x_2)$.
The lifted vector fields  are defined as follows
\begin{equation}\label{liftednotfrozen}
\tX_1=\p_{x_1}+(u(x)+s^2)\p_{x_2}, \ \tX_2=\e\p_{x_2},  \text{ and }\tX_3=\p_s.
\end{equation}
The $C^{1,\alpha}$ distribution $\{\tX_1,\tX_3\}$ is a step 3, bracket generating distribution since
\begin{equation}\label{commutatorsnotfrozen}
[\tX_1,\tX_3]=-2s\p_{x_2} \text{ and } [\tX_3,[ \tX_1,\tX_3]]=-2\p_{x_2}.
\end{equation}
The associated homogeneous dimension is $Q=5$.
 If  $x_0=((x_0)_1,(x_0)_2)\in \Om$ is a fixed point, then for all $x\in \Om$, $s\in \R$
 one can define exponential coordinates  $(\te_1(x,s), \te_2(x,s),\te_3(x,s))$, based at $(x_0,0)$, via the formula
 \begin{equation}\label{exponentialnotfrozen}
 (x,s)=\exp_{(x_0,0)}( \te_1(x,s)\tX_1+\te_2(x,s)\tX_2+\te_3(x,s)\tX_3).
 \end{equation}
 Here for any Lipschitz vector field  $Z$ in $\R^3$ we denote by $\exp_{(x_0,0)}(Z)$ the point
 $\gamma(1)$ where $\gamma$ is a curve such that $\gamma(0)=(x_0,0)$
 and $\gamma'(s)=Z(\gamma(s))$. The exponential coordinates can  be explicitly computed yielding 
 $$\te_1(x,s)=(x-x_0)_1, \ \e \te_2(x,s)=(x-x_0)_2- (x-x_0)_1 \big(\int_0^1u(\gamma(\tau)) d\tau- \frac{s^2}{3}\big), \text{ and }\ \te_3(x,s)=s.$$
 Note that if $d_E(x,x_0)=\sqrt{(x-x_0)_1^2+(x-x_0)_2^2}$ is the Euclidean distance in $\Omega$ then for $x$ sufficiently close
 to $x_0$ and for a certain constant $C>0$  (both depending on the $C^{1,\al}$ norm of $u$) one has that
 $$C^{-1} d_E(x,x_0) \le   \sqrt{\te_1(x,0)^2+ \e^2 \te_2(x,0)^2} \le C d_E(x,x_0).$$
Next, we define an analogue of the first order Taylor polynomial of $u$ as
\begin{equation}\label{P1}
P_{x_0}^1u \ (x)=u(x_0)+e_1(x) \tX_1 u(x_0,0) + e_2(x)\tX_2u(x_0,0),
\end{equation}
where $e_1(x)=\te_1(x,0)$ and 
\begin{equation}\label{e_2}
\e e_2(x)= (x-x_0)_2-  (x-x_0)_1  u(x_0).
\end{equation}
 We remark explicitly that
 \begin{multline}\label{e2notfrozen}
 \big|\e( \te_2(x,0)-e_2(x))\tX_2u(x_0,0)\big| = \big| \big( (x-x_0)_1 \big(\int_0^1 [u(\gamma(\tau))-   u(x_0) ]\  d\tau\big)\p_2 u(x_0)\big|
 \end{multline}

\paragraph{\bf Freezing. } At this point we introduce an appropriate  {\it freezing} of the vector fields by defining
\begin{equation}\label{frozenvectors}
X_{1,x_0}=\p_{x_1}+(P_{x_0}^1u (x) + s^2) \p_{x_2}, \ X_{2,x_0}=\e\p_{x_2} \text{ and }
X_{3,x_0}=\p_s.
\end{equation}
Observe that $\{X_{1,x_0},X_{3,x_0}\}$ is a distribution of smooth vector fields satisfying H\"ormander's finite rank hypothesis with step three. We denote by $d_{x_0}(\cdot , \cdot) $ the corresponding Carnot-Caratheodory
distance and remark that the homogeneous dimension of the space is 5. We also need the Riemannian distance function $d_{x_0,\e}(\cdot,\cdot)$ defined as 
the control distance associated to $\{X_{1,x_0}, X_{2,x_0}, X_{3,x_0}\}$.
 Define exponential coordinates $(\oe_1,\oe_2,\oe_3)$ in a neighborhood of $x_0$ through
the formula
$$(x,s)=\exp_{(x_0,0)}(\oe_1(x,s)X_{1,x_0} + \oe_2(x,s)X_{2,x_0}+\oe_3(x,s)X_{3,x_0}).$$
Note that $$\oe_1(x,0)= (x-x_0)_1, \ \oe_2(x,s)=\frac{1}{\e}\bigg( (x-x_0)_2 -(x-x_0)_1 \big( \int_0^1 P_{x_0}^1 u(\gamma(\tau)) d\tau -\frac{s^2}{3}\big)\bigg) \text{ and }\oe_3(x,0)=0.$$

It is well known (see for instance the discussion in \cite[Section 2.4]{cdpt:survey} that
$(\R^3,d_{x_0,\e})$ converges in the Gromov-Hausdorff sense to $(\R^3,d_{x_0})$. In particular one has that for each fixed $x$ and $s$,  then $d_{x_0,\e}((x,s),(x_0,0))\to d_{x_0}((x,s),(x_0,0))$ as $\e\to 0$. Moreover the  volume of the balls  $B_\e((x_0,0),R)$ in the $d_{x_0,\e}$ metric converges to the volume of the limit Carnot-Caratheodory balls, i.e.  $|B_\e((x_0,0),R)|\to |B_0((x_0,0),R)|$ as $\e\to 0$. In particular,
for $\e_0>\e>0$ sufficiently small there exists a constant $C>0$ depending only
on $\e_0$  such that 
\begin{equation}\label{stabilityofvolume} |B_\e((x_0,0),R)|\ge C R^5. \end{equation}
All this can also be seen explicitly in our special setting by 
observing that there exists a constant $C>0$ such that for $x$ near $x_0$, one has
(see \cite{NSW})
\begin{eqnarray}
C^{-1} d_{x_0,\e}((x,s),(x_0,0)) && \le \sqrt{ \oe_1^2(x,s)+\min(\oe_2^2(x,s),\ (\e \oe_2(x,s))^{2/3} \ )+\oe_3^2(x,s)}\le C d_{x_0,\e}((x,s),(x_0,0)). \notag \\
C^{-1} d_{x_0}((x,s),(x_0,0)) && \le \bigg( \oe_1^6(x,s)+(\e\oe_2)^2(x,s)+\oe_3^6(x,s)\bigg)^{1/6}
\le C d_{x_0}((x,s),(x_0,0)). \notag
\end{eqnarray}
Recall also that for $x$ and $x_0$ sufficiently close, $\e_0>\e>0$ sufficiently small
there exist positive constants $C_1,C_2$ depending only on $\e_0$ and 
on the $C^{1,\alpha}$ norm of $u$ such that
\begin{equation}\label{de-dcc}d_E((x,s),(x_0,0)) \le C_1 d_{x_0,\e}((x,s),(x_0,0))\le C_2 d_{x_0}((x,s),(x_0,0)).\end{equation}
Since $|\oe_2(x,s)|\to \infty$ as $\e\to 0$ one has that for a fixed $(x,s)$ then
$\lim_{\e\to 0} d_{x_0,\e}((x,s),(x_0,0))\approx d_{x_0}((x,s),(x_0,0)).$

In the following we will denote by $d_{x_0,\e}(x,x_0)$  the quantity $d_{x_0,\e}((x,0),(x_0,0))$.
\begin{lemma}\label{d1alpha}
If $u\in C^{1,\alpha}_E(\Om)$ there exists  constant $\e_0,C>0$ 
and a neighborhood $U$ of $x_0$ depending only on the $C^{1,\alpha}_E$ norm of $u$  such that
  for all $x$ in a sufficiently small neighborhood of $x_0$, and for all $\e_0>\e>0$,
  \begin{equation}\label{1alpha1}
  |u(x)-P^1_{x_0} u(x)| \le C d^{1+\al}_{x_0,\e}(x,x_0).
  \end{equation}
\end{lemma}

\begin{proof} 
Fix the points $x$ and $x_0$ and define the $C^{1,1}$ planar curve
$$\gamma(t)=\exp_{x_0}\bigg( t \big(\oe_1(x,0)X_{1,x_0}+\oe_2(x,0)X_{2,x_0}\big)\bigg),$$
so that $\gamma(1)=x$ and $\gamma(0)=x_0$.   From the mean-value theorem, for all $t\in (0,1)$ we can find
$\tilde t\in (0,t)$ such that
$$u(\gamma(t))-u(x_0)=\oe_1(x,0)X_{1,x_0}(\gamma(\tilde t))+\oe_2(x,0)X_{2,x_0}(\gamma(\tilde t)).$$ Hence for all $t\in (0,1)$ one has
\begin{eqnarray}\label{loc-1alpha-1}
u(\gamma(t))-P_{x_0}^1u(\gamma(t)) && = u(\gamma(t)) - u(x_0) = e_1(\gamma(t))\tX_1u(x_0,0)- e_2(\gamma(t))\tX_2u(x_0,0) \notag \\
&&= \sum_{i=1}^2  \bigg(\oe_i(x,0) X_{i,x_0}u(\gamma(\tilde t)) - e_i(\gamma(t)) \tX_i u(x_0,0)\bigg) \notag \\
&&=(x-x_0)_1 (u(x)-P_{x_0}^1u(x)) \p_{x_2}u(x) +
\oe_2(x,0) \big( X_{2,x_0}u(\gamma(\tilde t))-X_{2,x_0}u(x_0,0)\big) \notag \\
&& \ \ \ \ \ + \oe_2(x,0)\big(X_{2,x_0}u(x_0,0)-\tX_{2,x_0}u(x_0,0)\big)+
\big( \oe_2(x,0)-e_2(\gamma(t))\big) \tX_2 u(x_0,0).
\end{eqnarray}
Next, note that
\begin{eqnarray}
|\oe_2(x,0)-e_2(\gamma(t))| = && \frac{1}{\e} \bigg| -(x-x_0)_1\int_0^1 P_{x_0}^1u(\gamma(\tau)) d\tau -\frac{s^2}{3} +(\gamma(t)-x_0)_1 u(x_0)\bigg| \notag \\
&& \le \frac{1}{\e} \bigg| (x-x_0)_1 \int_0^1 ( P_{x_0}^1u(\gamma(\tau))- u(\gamma(\tau)) d\tau+\frac{s^2}{3}\bigg|  \notag \\ && \ \ \ \  \ \ + \frac{1}{\e} \bigg| \int_0^1 \big( (\gamma(t)-x_0)_1 u(x_0)-(x-x_0)_1 u(\gamma(\tau)) \big) d\tau\bigg| \notag \\
&& \le \frac{C}{\e} d_E^{1+\al}(x,x_0) + \frac{1}{\e} \bigg|  (x-x_0)_1 \int_0^1 ( P_{x_0}^1u(\gamma(\tau))- u(\gamma(\tau))  \ )d\tau\bigg|.
\end{eqnarray}
From the latter, from  \eqref{loc-1alpha-1} and  from observing that $X_{2,x_0}=\tX_{2,x_0}$, and 
$$|\oe_2(x,0) \big( X_{2,x_0}u(\gamma(\tilde t))-X_{2,x_0}u(x_0,0)\big)|\le C d_E^\al(x,x_0) |\e\oe_2(x,0)|$$ we obtain
\begin{multline}\label{loc-1alpha-2}|u(\gamma(t))-P_{x_0}^1u(\gamma(t))| \\ \le C(d_E^{1+\al}(x,x_0) + |\e\oe_2(x,0)| d_E^{\al}(x,x_0)+ |(x-x_0)_1| \int_0^1  | P_{x_0}^1u(\gamma(\tau))- u(\gamma(\tau)) | d\tau.\end{multline}
For $x$ sufficiently close to $x_0$ we have 
$$|\e\oe(x,0)| \le Cd_{x_0,\e}(x,x_0) \text{ and } |(x-x_0)|_1\le \frac{1}{2}.$$
Integrating \eqref{loc-1alpha-2} in the $t$ variable from $0$ to $1$ and using the latter
we obtain
$$\int_0^1  | P_{x_0}^1u(\gamma(\tau))- u(\gamma(\tau)) | d\tau \le C d_{x_0,\e}^{1+\al}(x,x_0).$$
Substituting this estimate back  in  \eqref{loc-1alpha-2}  we conclude the proof.
\end{proof}

\bigskip

\paragraph{\bf The frozen operators.} The freezing process described earlier  allows to introduce {\it 'frozen'} sub-Laplacians operators 
 $N_{\e, x_0}$ formally defined as $N_\e$, but in terms of the smooth vector fields $X^\e_{i, x_0}$
 instead of the original non-smooth vector fields $X_i$. Consider the operators

\begin{equation}\label{nondivfroz}N_{\e, x_0}z = \sum_{i, j=1}^{3}
a_{i j}(\nabla_\e u(x_0)) X_{i, x_0}X_{j, x_0}z,\end{equation}
where $a_{ij}$ are defined in (\ref{aij}), and

\begin{equation}\label{nondivfroz0}N_{x_0}z = \sum_{i,j=1 \atop i\not=2, j\not=2}^{3}
a_{i j}(\nabla_\e u(x_0)) X_{i, x_0}X_{j, x_0}z,\end{equation}

\bigskip

$N_{x_0}$ is an uniformly subelliptic operator with $C^{\infty}$
coefficients, and $N_{\e, x_0}$ can be considered as its elliptic regularization, 
with coefficients dependent on $\e$. The linear theory yields that both 
 $N_{x_0}$ and $ N_{\e, x_0} $  have fundamental
solutions $\G_{x_{0}}$ and $\G_{\e, x_{0}}$ respectively (see \cite{GT}, \cite{NSW} and \cite{BLU}). Since both $\G_{x_{0}}$ and  $\G_{\e, x_{0}}(\z,\x)$
depend on many variables, the notation
 $$X_{i, x_0}(\zeta_1)\G_{\e, x_{0}}( \; \cdot \; ,\x)$$
shall denote the $X_{i, x_0}$-derivative of $\G_{\e, x_{0}}(\zeta, \x)$
with respect to the variable $\zeta$, evaluated at the point $\z_1$.

Precise estimates for the fundamental solution  $\G_{x_{0}}$ have been established in  \cite{NSW} and \cite{BLU}, while in \cite{CiMa-F} it is proved that the fundamental 
solution $\G_{\e, x_{0}}$ of $ N_{\e, x_0} $ locally satisfies the same estimates as the limit 
kernel $\G_{x_{0}}$, with {\it  choice of  constants independent of} $\e$. These results can be summarized as follows;

\begin{theorem}\label{fundam}(\cite{CiMa-F}) Let $z_{0}\in\O$.
 For every compact set $K\subset \Omega\times (-1,1)$, for every $k\in \N$ 
 and for every multi-index $I=(i_1,...,i_k)$ with $i_j\in \{1,2,3\}$, there
  exist two positive
 constants $C, C_p$ independent of $\e$, such that
\begin {equation}\label{2}
|\nabla^I_{\e, x_0 }(\xi)\Gamma_{\e,x_0} (\cdot,\zeta)|\le C_{k }
\frac{d_{x_0,\e } ^{2-k}(\x,\zeta)}{|B_{\e} (\x, d_{x_0,\e}
(\x,\zeta))|}, 
\end {equation}
for every $\x, \z\in K$ with $\x\not= \z$,  where  $B_{\e}(\x,r)$
denotes the ball with center $\x$ and radius $r$ in the distance
$d_{x_0,\e}$, and  $\nabla^I_{\e, x_0 }=X_{i_1,x_0}...X_{i_n,x_0}$ denotes derivatives of order $|I|=k$
along the frozen vector fields $X_{i,x_0}$. If $k =0$ one intends that no derivative are applied on
$\Gamma_{\e, x_0}$.
\end{theorem}

\begin{remark}\label{notabene} Let $u$ be as in Lemma \ref{d1alpha}.
 Notice that if we set
$$\mathcal R=X_{1,u}-X_{1,x_0}= (u(x)-P_{x_0}^1u(x)-s^2)\p_{x_2},$$
then since $\p_{x_2}$ has order three, and in view of Lemma \ref{d1alpha}, the operator $\mathcal R$ has order $2-\alpha$ at the point $x_0$ (in the sense of \cite{fol:1975}).
Consequently, the  estimates in \eqref{2} do not continue to  hold
 if the derivatives along the frozen vector fields $\nabla^I_{\e x_0 }$, $|I|=2$,
are substituted by derivatives along $X_{i,u}^\e$ evaluated at the point $\xi=x_0$.
\end{remark}


%

Since the measure of the Ball is doubling, with doubling constant independent of $\e$, then
the space $(\Om\times(-1,1), d_{x_0,\e},dx)$ is a space of  homogenous type and the 
the following version of the 
fractional integration  theorem holds, (see for instance   \cite{cdg:isoperimetric})
\begin{proposition}\label{singularintegral} 
If  $K$ is a fixed compact set $K\subset\subset \Omega\times (-1,1)$, if 
$f\in L^q(\Om\times (-1,1))$  supported in $K$, and for each $(x,s)\in  \Omega\times (-1,1)$ we set
\begin{equation}\label{fio}
 I_p(f)(x,s) = \int\limits \frac{d_{x_0,\e}^{{p}}((x,s), (\z,\sigma))}{|B_\e((x,s), d_{x_0,\e}((x,s),( \z,\sigma))|}  f(\z,\sigma)  d\z d\sigma
 \end{equation}
then  there  exists a constant $C$, 
depending on $K,r,q$ but independent of $\e$ such that,
$$||I_p(f)||_{L^r( \Omega\times (-1,1))}\leq C ||f||_{L^q( \Omega\times (-1,1))},$$
and where $5-pq>0$ and $r= \frac{5q}{5-pq}$.
\end{proposition}
\begin{corollary} \label{corollarioscemo} Let $f\in C^{\infty}_0(\Om)$ and  extend it to a function on $ \Omega\times (-1,1)$ by setting $f(x,s):=f(x)$. Let $\mathcal K:(\Om\times(-1,1))^2 \to \R$ be a kernel satisfying
$$|\mathcal K((x,s),(\z,\sigma)) | \le C \frac{d_{x_0,\e}^{{p}}((x,s), (\z,\sigma))}{|B_\e((x,s), d_{x_0,\e}((x,s),( \z,\sigma))|} .$$ Set $u(x,s)$ to be defined
as $$u(x,s)=\int\mathcal K((x,s),(\z,\sigma)) f(\xi) d\z d\sigma.$$
If $u(x,s)=u(x)$ for all values of $s$ then 
$$||u||_{L^r( \Omega)}\leq C ||f||_{L^q( \Omega)},$$ with $r$ and $q$ as in the previous proposition.  
\end{corollary}
As a  consequence of the definition of fundamental solution one has the following representation formula:
\begin{proposition}\label{p402r}
Let   $  u$ be a fixed smooth function defined in $\O,$  let $N_{\e,u}$ be as in \eqref{nondivlin}, and let $z$ be a classical solution of $N_{\e,
u}z = g \in C^\infty(\O).$ Extend both $u$ and $z$ to be functions defined on $\Om\times (-1,1)$ by letting them be constant along the $s$ variable.
For  any $\phi\in C^\infty_0(\O\times(-1,1))$,  $\xi\in \O$ and $s\in (-1,1)$,
   the product  $z(\xi)\phi (\xi,s)$ can be represented as

\begin {equation}\label{prima}
  \begin {split}
   z(\xi)\phi(\x,s)  &
    =   \int\limits_{\O\times(-1,1)} \G((\x,s),(\z,\sigma))
   \left( z \,N_{\e x_0}\phi + \sum_{ij=1}^{2}\bar a_{ij}(x_0)
   \Big(X_{i,z_0 }z X_{j,z_0 }\phi + X_{j,z_0 }z X_{i,z_0}\phi\Big)\right) d\z\\
   &+  \int\limits_{\O\times(-1,1)} \G_{\e, x_{0}}((\x,s),(\z,\sigma))
\  g (\z)\, \phi(\z,\s)  d\z d\sigma  +
 \\
   & +  \sum_{ij=1}^{2} \int\limits_{\O\times(-1,1)} \G_{\e, x_{0}}((\x,s),(\z,\sigma))
\Big(\bar a_{ij}(x_0)- \bar a_{ij}(\z)\Big)X_{i, u } X_{j, u }z(\z)
   \, \phi(\z)  d\z d\sigma
\\
   &
-   \sum_{j=1}^{2}\bar a_{1 j}(x_0)\int\limits_{\O\times(-1,1)}
\G_{\e ,x_{0}}((\x,s),(\z,\sigma))\Big( u(\zeta)- P^1_{x_0} u(\zeta) -
\s^2\Big)\p_{2}X_{j,  u}z(\z)
   \, \phi(\z,\s)  d\z d\sigma
 \\
   &  + \sum_{i=1}^{2}\bar a_{i 1}(x_0)\int\limits_{\O\times(-1,1)}  X_{i, x_0}(\x,s)\G_{\e, x_{0}}(\cdot,(\z,\sigma))
 \Big( u(\zeta)- P^1_{x_0} u (\z) - \s^2\Big)\p_{2} z(\z)
   \, \phi(\z,\s)  d\z d\sigma
 \\
   &  + \sum_{i=1}^{2}\bar a_{i 1}(x_0)\int\limits_{\O\times(-1,1)}  \G_{\e, x_{0}}((\x,s),(\z,\sigma))
 \Big( u(\zeta)- P^1_{x_0} u (\z) - \s^2\Big)\p_{2} z(\z)
   \, X_{i,x_0}\phi(\z,\s)  d\z d\sigma.
\end{split}\end{equation}
In order to simplify notations we have set
$$\bar a_{ij}(\z) = a_{ij}(\nabla_\e u(\z)).$$
\end{proposition}

\begin{proof}
 In view of the definition of fundamental solution, we have
\begin {equation}\label{prima1}
  \begin {split}
   z(\xi)\phi(\x,s)   & =\!\!\!\!\!  \int\limits_{\O\times(-1,1)} \G_{\e, x_{0}}((\x,s),(\z,\sigma))
   N_{\e x_0}(z\phi)(\z,\sigma) d\z d\sigma\\
   & = \!\!\!\!\!   \int\limits_{\O\times(-1,1)} \G_{\e, x_{0}}((\x,s),(\z,\sigma))
   \left( z \,N_{\e x_0}\phi + \sum_{ij=1}^{2}\bar a_{ij}(x_0)
   \Big(X_{i,x_0 }u X_{j,x_0 }\phi + X_{j,x_0 }z X_{i,x_0}\phi\Big)\right) d\z d\sigma\\
   &+  \!\!\!\!\! \int\limits_{\O\times(-1,1)} \G_{\e, x_{0}}((\x,s),(\z,\sigma))
   N_{\e u} z (\z)\, \phi(\z,\s)  d\z d\sigma  \\
&+ \!\!\!\!\! \int\limits_{\O\times(-1,1)} \G_{\e, x_{0}}((\x,s),(\z,\sigma))
   (N_{\e x_0}  - N_{\e u} )z (\z, \sigma)\, \phi(\z,\s)  d\z d\sigma.
 \end{split}\end{equation}
 
Let us now compute the difference between the operator $N_{\e 
u}$ and its frozen counterpart. We will emphasize the presence of the variable $\sigma$
in those terms where the coefficients of the vector fields involve that variable.
 \begin {equation}\label{LmenoL}
 \begin {split}
  \bigg((N_{\e x_0} - & N_{\e u})z\bigg)(\z, \sigma) =
\sum_{ij=1}^{2}\Big(\bar  a_{ij}(x_0)-
\bar a_{ij}(\z)\Big)X_{i, u } X_{j, u }z(\z) -
\\
   &-\sum_{ij=1}^{2}\bar a_{ij}(z_0)\Big(X_{i, u }
X_{j, u }-X_{i,x_0 } X_{j,x_0 }\Big) z(\z,\sigma) =
\\
& = \sum_{ij=1}^{2}\Big(\bar a_{ij}(x_0)-
\bar a_{ij}(\z)\Big)X_{i, u } X_{j,u }z(\z) -
\\&-\sum_{ij=1}^{2}\bar a_{ij}(x_0)\Big((X_{i, u }-X_{i,x_0
   })X_{j, u}+X_{i,x_0 }(
X_{j, u }-X_{j,x_0 }) \Big) z(\z,\s)
\\
& = \sum_{ij=1}^{2}\Big(\bar a_{ij}(x_0)-
\bar a_{ij}(\z)\Big)X_{i, u } X_{j, u }z(\z) -
\\
 &-\sum_{ij=1}^{2}\bar a_{ij}(x_0)\Big(\delta_{i 1}(u(\zeta)- P^1_{z_0}
u(\zeta)-\s^2)\p_{2}X_{j,u} +
X_{i,x_0}(\delta_{j,1}( u(\zeta)- P^1_{x_0}
u(\zeta) - \s^2)\p_{2})\Big)z(\z)
\\& =
\sum_{ij=1}^{2}\Big(\bar a_{ij}(x_0) - \bar a_{ij}(\z)\Big)X_{i, u
} X_{j, u }z(\z)
\\&
-\sum_{i=1}^{2}\bar a_{i 1}(x_0)(u(\zeta)- P^1_{x_0}
u(\zeta) -\s^2)\p_{2}X_{i,  u}z(\z)
\\&
- \sum_{i=1}^{2}\bar a_{i 1}^\e(x_0)X_{i,x_0}\Big(( u(\zeta)-
P^1_{x_0}u (\z)- \s^2)\p_{2}\Big)z(\z).
\end{split}\end{equation}
The integral of the  last   term  in (\ref{LmenoL}) becomes
\begin{equation}\label{formula3}
\begin{split}
 & \sum_{i=1}^{2}  \bar a_{i 1}(x_0)\int\limits_{\O\times(-1,1)}    \G_{\e, x_{0}}((\x,s),(\z,\sigma)) X_{i, x_0}\Big(( u(\zeta)- P^1_{x_0}u(\zeta) - \s^2)\p_{2}\Big)z(\z)
\phi(\z,\s)  d\z d\s
\end{split}
\end{equation}
(integrating by part)
\begin{equation}
\begin{split}
  = &-  \sum_{i=1}^{2} \bar a_{i 1}(x_0) \int\limits_{\O\times(-1,1)}   X_{i,x_0}(\x,s) \G_{\e, x_{0}}(\cdot,(\z,\sigma)) \,
( u(\zeta)- P^1_{x_0}u(\zeta)- \s^2)\p_{2}  z(\z)\, \phi(\z,\s) d\z d\s \\
 &-  \sum_{i=1}^{2}  \bar a_{i 1}(x_0)\int\limits_{\O\times(-1,1)}    \G_{\e, x_{0}}((\x,s),(\z,\sigma)) \,
(u(\zeta)- P^1_{x_0} u(\zeta)- \s^2)\p_{2}  z(\z)\,
X_{i,x_0}\phi(\z,\s) d\z d\s
\end{split}
\end{equation}
Inserting all terms in the preceeding formula we conclude the proof. 
\end{proof}
 
 \begin{lemma}\label{L-comparing}
 Let $u$ and $x_0\in \Om$ be as above. There exists a neighborhood $U$ of $(x_0,0)$,
 possibly depending on $\e$, such that for all   $(\xi,s)\in U$ one has
 \begin{multline}\label{comparing}
\bigg| \bigg(X_{k,x_0}X_{l,x_0} (\xi,s) -  X_{k,x_0}X_{l,x_0} (x_0,0) \bigg) \nabla_{x_0,\e} \Gamma_{\e, x_{0}}(\cdot,(\z,\sigma))\bigg|\\  \le  d_{x_0}((\xi,s),(x_0,0))
\Bigg( \big| \nabla_{x_0,\e}^I \Gamma_{\e, x_{0}}((\x,s),(\z,\s))\big| 
+ \big| \nabla_{x_0,\e}^I \Gamma_{\e, x_{0}}((x_0,0),(\z,\s))\big|\Bigg), \text{ with }|I|=4.
 \end{multline}
 \end{lemma}
 \begin{proof} The proof follows from the mean value principle: 
 Set $L= d_{x_0}((\xi,s),(x_0,0))$. For every $\delta>0$ consider 
 a horizontal curve  $\gamma:[0,L+\delta]\to \Om\times (-1,1)$, parametrized by arc-length and joining
 $(\x,s)$ to $(x_0,0)$. For every $f\in C^1$ one has
 $$|f(\x,s)-f(x_0,0)|\le  \int_0^L |\frac{d}{dt} f(\gamma(t))| dt = \int_0^L|\langle \nabla f(\gamma(t)), \gamma'(t)\rangle |dt \le (L+\delta)  \sup_{\gamma} |\nabla_{x_0,\e} f|.$$
There exists neighborhood  $(x_0,0)\in U$ depending on the $C^1$ norm of $f$ for which  we have
 $|\nabla_{x_0,\e} f(\z,\s)| \le 2 ( |\nabla_{x_0,\e} f(x_0,0)|+ |\nabla_{x_0,\e} f(\x,s)|)$
 for all $(\z,\s)\in U$. The lemma now follows from choosing$$f=X_{k,x_0}X_{l,x_0}X_{i,x_0} \Gamma_{\e,x_0}(\cdot, (\z,\s))$$ and  observing that the smoothness of $\Gamma_{\e,x_0}$ depends on $\e>0$.
 \end{proof}
 \begin{lemma}\label{repder}Let   $z$ be a smooth  solution of $N_{\e, u}z = 0.$  For any $s\in (-1,1)$ and $x_0\in \Om$, one can represent the 'frozen' second order  horizontal
derivatives of $z$ at $x_0$ as
\begin {equation}\label{prima2}
  \begin {split}
X_{k,x_0} & X_{ l,x_0} (z (x_0) \phi(x_0,s))  =\\
\!  & =   \int\limits_{\O\times(-1,1)} X_{k,x_0} X_{l,x_0} (x_0,s) \G_{\e, x_{0}}(\cdot,(\z,\sigma))
   \left( z \,N_{\e x_0}\phi + \sum_{ij=1}^{2}\bar a_{ij}(x_0)
   \Big(X_{i,z_0 }z X_{j,z_0 }\phi + X_{j,z_0 }z X_{i,z_0}\phi\Big)\right) d\z\\
   & +  \sum_{ij=1}^{2} \int\limits_{\O\times(-1,1)} X_{k,x_0} X_{l,x_0} (x_0,s) \G_{\e, x_{0}}(\cdot,(\z,\sigma))
\Big(\bar a_{ij}(x_0)- \bar a_{ij}(\z)\Big)X_{i, u } X_{j, u }z(\z)
   \, \phi(\z,\s)  d\z d\sigma
\\
   &
-   \sum_{j=1}^{2}\bar a_{1 j}(x_0)\int\limits_{\O\times(-1,1)}
 X_{k,x_0} X_{l,x_0} (x_0,s) \G_{\e, x_{0}}(\cdot,(\z,\sigma))\Big( u(\zeta)- P^1_{x_0} u(\zeta) -
\s^2\Big)\p_{2}X_{j,  u}z(\z)
   \, \phi(\z,\s)  d\z d\sigma
 \\
   &  + \sum_{i=1}^{2}\bar a_{i 1}(x_0)\int\limits_{\O\times(-1,1)}  X_{k,x_0} X_{l,x_0}   X_{i, x_0} (x_0,s)\G_{\e x_{0}}(\cdot,(\z,\sigma))
 \Big( u(\zeta)- P^1_{x_0} u (\z) - \s^2\Big)\p_{2} z(\z)
   \, \phi(\z,\s)  d\z d\sigma
 \\
   &  + \sum_{i=1}^{2}\bar a_{i 1}(x_0)\int\limits_{\O\times(-1,1)} X_{k,x_0} X_{l,x_0}  (x_0,s) \G_{\e, x_{0}}(\cdot,(\z,\sigma))
 \Big( u(\zeta)- P^1_{x_0} u (\z) - \s^2\Big)\p_{2} z(\z)
   \, X_{i,x_0}\phi(\z,s)  d\z d\sigma.
\end{split}\end{equation}
 \end{lemma}
 \begin{proof}
 Since the proof is similar to that of \cite[Proposition 3.9]{CPP}, we only sketch the argument for the {\it most singular} term in the representation formula, i.e.
 \begin{equation}\label{I-10/3}
I((\x,s), (x_0,0))= \sum_{i=1}^{2}\bar a_{i 1}(x_0)\int\limits_{\O\times(-1,1)}  X_{i, x_0}(\x,s)\G_{\e x_{0}}(\cdot,(\z,\sigma))
 \Big( u(\zeta)- P^1_{x_0} u (\z) - \s^2\Big)\p_{2} z(\z)
   \, \phi(\z,s)  d\z d\sigma.
 \end{equation}
 We want to show that 
 \begin{equation}\label{concl-10/3}
 X_{k,x_0} X_{l,x_0} (x_0,0) I (\cdot, (x_0,0))=I^{(2)},
 \end{equation}
  where
 \begin{equation}\label{pre10/3}
I^{(2)}(x_0)=\sum_{i=1}^{2}\bar a_{i 1}(x_0)\int\limits_{\O\times(-1,1)}  X_{k,x_0} X_{l,x_0} (x_0,s) X_{i, x_0}\G_{\e x_{0}}(\cdot,(\z,\sigma))
 \Big( u(\zeta)- P^1_{x_0} u (\z) - \s^2\Big)\p_{2} z(\z)
   \, \phi(\z,s)  d\z d\sigma 
   \end{equation}
   Note that the latter is well defined in view of the estimates \eqref{1alpha1} and \eqref{2}.
   To show \eqref{concl-10/3} we consider a family of smooth test functions $\chi_{x_0,\e}(
   (\xi,s), (\z,\sigma))$, satisfying  for some choice of $C>0$ and for small $\e>0$, 
   (i) $0\le \chi_{x_0,\e}(
   (\xi,s), (\z,\sigma)) \le 1 $, (ii) $\chi_{x_0,\e}((
   (\xi,s), (\z,\sigma))=0$ if $d_{x_0,\e}((\xi,s), (\z,\sigma)) \le 2 C\e$, (iii)
    $\chi_{x_0,\e}(
   (\xi,s), (\z,\sigma))=1$ if $d_{x_0,\e}((\xi,s), (\z,\sigma)) \ge 4 C\e$, (iv)
   $|\nabla_{\e, x_0}^I \chi_{x_0,\e}((
   (\xi,s), (\z,\sigma))| \le C \e^{-|I|}$ for all multi-indeces $I$,
   for some choice of $C>0$ and for small $\e>0$. For the existence
   of such function see \cite{CGL} and note that the construction argument in that paper uses 
   only the estimates on the fundamental solution. Define the smooth approximation
   \begin{multline}I_\e((\x,s), (x_0,0))= \sum_{i=1}^{2}\bar a_{i 1}(x_0)\int\limits_{\O\times(-1,1)}  X_{i, x_0} (\x,s)\G_{\e x_{0}}(\cdot,(\z,\sigma))
  \\ \cdot \ \Big( u(\zeta)- P^1_{x_0} u (\z) - \s^2\Big) \ \chi_{x_0,\e}(
   (\xi,s), (\z,\sigma)) \ \p_{2} z(\z)
   \, \phi(\z,s)  d\z d\sigma.\end{multline}
  Once we establish the bounds
   \begin{align}\label{hypothesis-3.2}
   \bigg| I_\e ((\x,s), (x_0,0)) - I ((\x,s), (x_0,0))\bigg| & \le C \e^{2+\al} \notag \\
   \bigg|  X_{k,x_0} X_{l,x_0} (\xi,s)  I_\e (\cdot, (x_0,0)) - I^{(2)}(x_0)\bigg| & \le C \e^{\al},
   \end{align}
   for\footnote{In view of the dependence of the neighborhood $U$ from $\e$ in Lemma \ref{L-comparing}, possibly one needs to work with a smaller scale $d_{x_0,\e}((\xi,s), (x_0,0) \le o( \e)$ depending on the $C^3$ norm of $\Gamma_{\e,x_0}$} $d_{x_0,\e}((\xi,s), (x_0,0) \le \e$, then invoking \cite[Proposition 3.2]{CPP} we
 immediately   conclude \eqref{concl-10/3}.
 To prove the first estimate in \eqref{hypothesis-3.2} we first observe that for $\e>0$ sufficiently small
\begin{multline}
  \bigg| I_\e ((\x,s), (x_0,0)) - I ((\x,s), (x_0,0))\bigg| \\ \le
  C\int_{d_{x_0,\e}((\z,\s), (\x,s) \le 4C\e}  \bigg(1-\chi_{x_0,\e}(
   (\xi,s), (\z,\sigma)) \bigg) \frac{d_{x_0,\e }( (\x,s),(\zeta,\sigma))}{|B_{\e} (\x, d_{x_0,\e}
( (\x,s),(\zeta,\sigma))|} \cdot  \\
\cdot d_{x_0,\e }((\zeta,\sigma), (x_0,0))^{1+\al} | \p_{2} z(\z)
   \, \phi(\z,s)|  d\z d\sigma.
\end{multline}
Since $$d_{x_0,\e }((\zeta,\sigma), (x_0,0)) \le C( d_{x_0,\e }((\zeta,\sigma),(\x,s)) +
d_{x_0,\e }((x_0,0), (\x,s))\le C(  d_{x_0,\e }((\zeta,\sigma),(\x,s)) +\e),$$ then we conclude
$$\bigg| I_\e ((\x,s), (x_0,0)) - I ((\x,s), (x_0,0))\bigg| \le C \int_{d_{x_0,\e}((\z,\s), (\x,s) \le 4C\e}  \frac{d_{x_0,\e }( (\x,s),(\zeta,\sigma))^{2+\al}}{|B_{\e} (\x, d_{x_0,\e}
( (\x,s),(\zeta,\sigma))|} \le C \e^{2+\al}.$$
Next we turn to the second  estimate in \eqref{hypothesis-3.2}. Observe that
$$\bigg|  X_{k,x_0} X_{l,x_0} (\xi)  I_\e ((\x,s), (x_0,0)) - I^{(2)}(x_0)\bigg| \le |A_1|+|A_2|+|A_3|,$$
where
\begin{multline}
A_1=\int\limits_{\O\times(-1,1)} \Bigg| \bigg( X_{k,x_0} X_{l,x_0} (\xi,s)-  X_{k,x_0} X_{l,x_0} (x_0,s) \bigg) \nabla_{x_0,\e} \G(\cdot, (\z,\sigma)) \cdot\\ \cdot
 \Big( u(\zeta)- P^1_{x_0} u (\z) - \s^2\Big) \ \chi_{x_0,\e}(
   (\xi,s), (\z,\sigma)) \ \p_{2} z(\z)
   \, \phi(\z,\s)  \Bigg| d\z d\sigma,
   \end{multline}
   \begin{multline}
A_2=\int\limits_{\O\times(-1,1)} \Bigg| X_{k,x_0} X_{l,x_0} (x_0,s)  \nabla_{x_0,\e} \G(\cdot, (\z,\sigma)) \cdot \\ \cdot
 \Big( u(\zeta)- P^1_{x_0} u (\z) - \s^2\Big) \ \bigg(1-\chi_{x_0,\e}(
   (\xi,s), (\z,\sigma)) \bigg)\ \p_{2} z(\z)
   \, \phi(\z,\s)  \Bigg| d\z d\sigma,
   \end{multline}
   and
   \begin{multline}
A_3=\int\limits_{\O\times(-1,1)} \Bigg|  X_{l,x_0}  \nabla_{x_0,\e}(\xi,s) \G(\cdot, (\z,\sigma))
\cdot \\ \cdot
 \Big( u(\zeta)- P^1_{x_0} u (\z) - \s^2\Big) \ X_{k,x_0}(\xi) \chi_{x_0,\e}(
   (\xi,s), (\z,\sigma)) \ \p_{2} z(\z)
   \, \phi(\z,\s)  \Bigg| d\z d\sigma
   \\
   + \int\limits_{\O\times(-1,1)} \Bigg|   \nabla_{x_0,\e} (\xi,s) \G(\cdot, (\z,\sigma))
\cdot \\ \cdot
 \Big( u(\zeta)- P^1_{x_0} u (\z) - \s^2\Big) \  X_{l,x_0} X_{k,x_0}(\xi,s) \chi_{x_0,\e}(
  \cdot, (\z,\sigma)) \ \p_{2} z(\z)
   \, \phi(\z,\s)  \Bigg| d\z d\sigma
   \end{multline}
  Invoking Lemma \ref{L-comparing}   one can complete the proof arguing as in \cite[page 734]{CPP}.  As usual, we examine in detail only the integral $A_1$  which
  contains the most singular integrand.
  Note that if $\z\in $supp$ \chi_{x_0,\e}(
   (\xi,s),\cdot) $ then $d_{x_0,\e}((\xi,s), (\z,\sigma)) \ge 2 C\e$, consequently
   $$d_{x_0} ((x_0,0), (\z,\sigma)) \ge \frac{1}{C} d_{x_0,\e}((\xi,s), (\z,\sigma))
   -d_{x_0,\e}((\xi,s), (x_0,0))\ge \e.$$ In view of Lemma  \ref{L-comparing}, \eqref{fundam}
   and Lemma \ref{d1alpha} we have 
    \begin{multline}
   |A_1| \le C \int_{d_{x_0} ((x_0,0), (\z,\sigma))\ge \e}   \!\!\!\!\!\!\!\!\!\!\!\!\!\!\!\!\!
   d_{x_0}((\xi,s),(x_0,0))  d_{x_0}^{1+\al}((x_0,0),(\z,\s))
 \phi(\z)  \frac{d_{x_0}((x_0,0),(\z,\s))^{-2}}{|B((x_0,0), d_{x_0}((x_0,0),(\z,\s)))|}
 d\z d\s \\
+ C \int_{d_{x_0} ((\x,s),(\z,\s))\ge \e}   \!\!\!\!\!\!\!\!\!\!\!\!\!\!\!\!\!
   d_{x_0}((\xi,s),(x_0,0))  d_{x_0}^{1+\al}((x_0,0),(\z,\s))
 \phi(\z)
 \frac{d_{x_0}((\x,s),(\z,\s))^{-2}}{|B((\x,s), d_{x_0}((\x,s),(\z,\s)))|}
   d\z d\s  \le C \e^{\al}.
   \end{multline}
 \end{proof}
 
Using the representation formula above,  the fractional integration result in
 Proposition \ref{singularintegral}
  and
 Corollary \ref{corollarioscemo} we  finally can proceed to the proof of  the main result of the section:

\begin{proof}[Proof of Theorem \ref{p402}] We will prove $(i)$ only. The proof of $(ii)$ follows along a similar argument.
Using the representation formula \eqref{prima2} one can represent the second horizontal
derivatives $ X_{k,x_0} X_{ l,x_0} (z(x_0)\phi(x_0,s)  $  of $z$ at any point $(x_0,s)\in \Om\times (-1,1)$ through integral operator with  kernels of the form 
$$ X_{k,x_0} X_{l,x_0} X_{i, x_0}(x_0,s)\G_{\e x_{0}}(\cdot,(\z,\sigma))
 \Big( u(\zeta)- P^1_{x_0} u (\z) - \s^2\Big)$$
 $$X_{k,x_0} X_{l,x_0}(x_0,s) \G_{\e, x_{0}}(\cdot,(\z,\sigma))
\Big(\bar a_{ij}(x_0)- \bar a_{ij}(\z)\Big)$$
and 
$$ X_{k,x_0} X_{l,x_0} (x_0,s) \G_{\e, x_{0}}(\cdot,(\z,\sigma))
 \Big( u(\zeta)- P^1_{x_0} u (\z) - \s^2\Big)$$
%

To establish the non-singular character of such kernels one needs to invoke
the  estimates on  the derivatives of the fundamental solution of the frozen operator
 $\G_{\e x_{0}}$   in Theorem \ref{fundam}. To  estimate the $L^{10/3}$ norm of each term
in the right-hand side of \eqref{prima2} one  uses  the fractional integral 
estimates in Corollary \ref{corollarioscemo}. The 'worst' possible term is the one corresponding to three derivatives on $\G$, i.e.
\begin{multline}\label{10/3}
\sum_{i=1}^{2}\bar a_{i 1}(x_0)\int\limits_{\O\times(-1,1)}  X_{k,x_0} X_{l,x_0} X_{i, x_0}(x_0,s)\G_{\e x_{0}}(\cdot,(\z,\sigma))
 \Big( u(\zeta)- P^1_{x_0} u (\z) - \s^2\Big)\p_{2} z(\z)
   \, \phi(\z)  d\z d\sigma \\
   =\int \mathcal K( (x_0,s), (\z,\sigma))\p_{2} z(\z)
   \, \phi(\z)  d\z d\sigma, \end{multline}
   with (in view of Remark \ref{notabene})
   $$|\mathcal K( (x_0,s), (\z,\sigma))| \le C \frac{d_{x_0,\e}^{{\alpha}}((x_0,s), (\z,\sigma))}{|B_\e((x,s), d_{x_0,\e}((x_0,s),( \z,\sigma))|} .$$
  Note that  the expression in \eqref{10/3} does not depend on $s$. Moreover since 
   from the assumptions one has $p>10/3>$, and hence $p>50/(15+ 10 \alpha)$ then Corollary \ref{corollarioscemo} yields immediately that the integral in \eqref{10/3} is in $L^r$ with
   $$r=\frac{5p}{5-p\alpha} >\frac{10}{3}.$$
   The rest of the terms in the right-hand side of  \eqref{prima2} are estimated similarly.
   
 At this point we have proved that the function
   $x_0\to X_{k,x_0} X_{l,x_0} z (x_0)$ is in $  L^{\frac{10}{3}}(K).$
 Next  we observe that 
   $$X_{k,u} X_{l,u} (x_0) z - X_{k,x_0} X_{l,x_0} (x_0,0) z =\begin{cases}
 0 & \text{ if } l=2 \text{ and }k=1,2 \\
 u (x_0) \p_{x_2}u(x_0) \p_{x_2}  z (x_0) &\text{ if } l=1 \text{ and } k=1\\
0 &\text{ if } l=1 \text{ and } k=2
 \end{cases}.$$
 In view of the hypothesis $\p_{x_2}z\in L^p(K)$ then one finally concludes
 $$X_{k,u} X_{l,u}  z - X_{k,x_0} X_{l,x_0} (\cdot,0) z \in L^p(K)$$
 with $p>10/3$.
\end{proof}

\section{Cacciopoli type inequalities: $W^{m,p}_{\e,\loc}$ a priori estimates.}

In this section we prove a {\it a priori} estimate, in the
Sobolev spaces $W^{m,p}_{\e, \loc}(\Omega)$ for solutions $z$  of
equation (\ref{eq111}), under the assumption that
$u$ is  smooth and satisfies \eqref{MM}.

The starting point is a  Caccioppoli type inequality for
derivatives of solution of  (\ref{eq111}) in the directions $X_i$.
By Lemmas \ref{lem4} and \ref{lem6} these derivatives solve the same equation 
\eqref{eq111}, with a different second member, hence we will focus on this PDE.

\begin{lemma}\label{lem1} Assume that $f_0\in
L^1_{loc}(\Omega)$, and that $z\in W^{2,2}_{\e, loc}(\Omega)\cap
W^{1,3}_{1, loc}(\Omega)$ is a solution of equation (\ref{eq111})
For every $p\geq 3$ there exist
constants $C_1=C_1(p,M) $  with the constant $M$ as in
(\ref{MM}) and independent of $\e$ and $z$ such that for every
non-negative $\phi \in C^\infty_0(\Omega),$
 we have $$ \int |\nabla_\e(|z|^{(p-1)/2})|^2
\phi^2 \leq  C_1 \left(\int |z|^{p-1} \big( \phi^2 + |\nabla_\e
\phi|^2\big)  - \int f |z|^{p-3}z \phi^2\right).
$$

\end{lemma}

\begin{proof} Let us multiply both members of equation (\ref{eq111})
by $|z|^{p-3}z \phi^2$, and integrate. We obtain
\begin{equation} \label{s}
\int f  |z|^{p-3}z \phi^2 = \int
X_{i}\Big(\frac{a_{ij}(\nabla_\e u)}{\sqrt{1 + |\nabla_\e
u|^2}}X_jz  \Big) |z|^{p-3}z \phi^2 =\end{equation}(since $X_1^* =
-X_1 -\p_2u $, $X_2^* = - X_2$)
 $$ = - \int \p_2u \delta_{i1}\frac{a_{ij}(\nabla_\e u)}{\sqrt{1 +
 |\nabla_\e
u|^2}}X_jz  |z|^{p-3}z \phi^2 -$$$$ - (p-2) \int
\frac{a_{ij}(\nabla_\e u)}{\sqrt{1 + |\nabla_\e u|^2}}X_jz X_iz
|z|^{p-3} \phi^2 - 2 \int \frac{a_{ij}(\nabla_\e u)}{\sqrt{1 +
|\nabla_\e u|^2}}X_jz |z|^{p-3}z  \phi X_i \phi.
$$
This obviously
implies that there exists a constant $C>0$ such that

\begin{equation}\label{eq40}
\frac{4(p-2)}{(p-1)^2} \int |\nabla_\e (|z|^{(p-1)/2})|^2 \phi^2
\leq C\int |z|^{p-1}(|\nabla_\e\phi|^2+ \phi^2)  - \int f
|z|^{p-3}z \phi^2,
\end{equation}
concluding the proof.
\end{proof}

\bigskip

\begin{lemma}\label{lemmalteo1}
Let $p\geq 3$ be fixed and $u$ be a function satisfying the bound \eqref{MMMM= }. Let $f \in
C^\infty(\Omega)$, and  let $z$ be a smooth solution of equation
(\ref{eq111}). There exist two constants $C$ and $\tilde C$
which depend on $p$ and the constant $M$ in (\ref{MM}) but are
independent of $\e$ and $z$ such for every $\phi\in
C^\infty_0(\Omega),$ $\phi>0,$
\begin{equation}\int |\nabla_\e ( |\nabla_\e z| ^{(p-1)/2})|^2
\phi^{2p}\leq C \int \big(|\nabla_\e\phi|^{2}+\phi^2\big)^p
+\end{equation}$$+
 \int |\nabla_\e z|^{p+1/2} \phi^{2p} 
+ \int  |X_2( \p_2 u)|^p  \phi^{2p} + \int |f| ^{(2p+1)/7} \phi^{2p} + \int |f| ^{(2p+1)/5}(|\nabla_\e \phi| + \phi)^{(2p+1)/5} \phi^{(8p-1)/5}  +$$
$$+ \int|\nabla^2_\e u| |\nabla_\e z|^{p-1} \phi^{2p} 
+ \int|\nabla^2_\e u|^2 |\nabla_\e z|^{p-1} \phi^{2p} + 
\int|\nabla^2_\e u| |\nabla_\e z|^{p-1} \phi^{2p-1}|\nabla_\e \phi| .$$

\end{lemma}

\begin{proof} Since $z$ is a solution of equation (\ref{eq111})
then by Lemma \ref{lem4}
 $s_1 = X_1 z $ satisfies equation 
$$ M_\e s_1 = X_1(f ) + X_{i}\Bigg(\frac{a_{i2}(\nabla_\e u)}{\sqrt{1
+ |\nabla_\e u|^2}} \p_2 u X_2z \Bigg) -$$$$-
X_{i}\Bigg(X_1\Big(\frac{a_{ij}(\nabla_\e u)}{\sqrt{1 + |\nabla_\e
u|^2}}\Big) X_jz \Bigg)
 + \p_2 u  X_2\Bigg(\frac{a_{2j}(\nabla_\e u)}{\sqrt{1 +
|\nabla_\e u|^2}}X_jz \Bigg)$$

 Using Lemma
\ref{lem1} we deduce \begin{equation}\label{42}\int |\nabla_\e (
|s_1| ^{(p-1)/2})|^2 \phi^{2p}  \leq C_1\int
|s_1|^{p-1}\big(|\nabla_\e\phi|^{2}+\phi^2\big)\phi^{2p-2} -
\end{equation}
$$- \int  |s_{1}|^{p-3}s_1
X_{i}\Bigg(\frac{a_{i2}(\nabla_\e u)}{\sqrt{1 + |\nabla_\e u|^2}}
\p_2 u X_2z \Bigg)\phi^{2p} + $$$$+\int  |s_{1}|^{p-3}s_1
X_{i}\Bigg(X_1\Big(\frac{a_{ij}(\nabla_\e u)}{\sqrt{1 + |\nabla_\e
u|^2}}\Big) X_jz \Bigg)\phi^{2p} $$$$
 - \int  |s_{1}|^{p-3}s_1  \p_2 u  X_2\Bigg(\frac{a_{2j}(\nabla_\e u)}{\sqrt{1 +
|\nabla_\e u|^2}}X_jz \Bigg)\phi^{2p} - \int |s_{1}|^{p-3}s_1 X_1 f \phi^{2p}=$$ (integrating by part all
terms in the right hand side)
$$
=  C_1\int
|s_1|^{p-1}\big(|\nabla_\e\phi|^{2}+\phi^2\big)\phi^{2p-2} + \int
|s_{1}|^{p-3}s_1  \frac{a_{12}(\nabla_\e u)}{\sqrt{1 +
|\nabla_\e u|^2}} (\p_2 u)^2 X_2z \phi^{2p} + $$
$$+ (p-2) \int
|s_{1}|^{p-3} X_is_1\frac{a_{i2}(\nabla_\e u)}{\sqrt{1 +
|\nabla_\e u|^2}} \p_2 u X_2z \phi^{2p}+$$$$ + 2p \int
|s_{1}|^{p-3}s_1 \frac{a_{i2}(\nabla_\e u)}{\sqrt{1 + |\nabla_\e
u|^2}} \p_2 u X_2z  \phi^{2p-1}X_i\phi + $$

$$- \int
|s_{1}|^{p-3}s_1 \p_2 u  X_1\Big(\frac{a_{1j}(\nabla_\e
u)}{\sqrt{1 + |\nabla_\e u|^2}}\Big) X_jz  \phi^{2p} $$
$$- (p-2)\int
|s_{1}|^{p-3} X_{i}s_1  X_1\Big(\frac{a_{ij}(\nabla_\e u)}{\sqrt{1
+ |\nabla_\e u|^2}}\Big) X_jz  \phi^{2p} $$
$$- 2p\int
|s_{1}|^{p-3}s_1  X_1\Big(\frac{a_{ij}(\nabla_\e u)}{\sqrt{1 +
|\nabla_\e u|^2}}\Big) X_jz  \phi^{2p-1} X_i\phi $$

$$
+ (p-2) \int  |s_{1}|^{p-3}X_2 s_1  \p_2 u \frac{a_{2j}(\nabla_\e
u)}{\sqrt{1 + |\nabla_\e u|^2}}X_jz \phi^{2p}+$$
$$
+ \int  |s_{1}|^{p-3}s_1 X_2 \p_2 u  \frac{a_{2j}(\nabla_\e
u)}{\sqrt{1 + |\nabla_\e u|^2}}X_jz \phi^{2p} + $$
$$
+ 2p \int  |s_{1}|^{p-3}s_1  \p_2 u \frac{a_{2j}(\nabla_\e
u)}{\sqrt{1 + |\nabla_\e u|^2}}X_jz \phi^{2p-1} X_2\phi + 
$$$$+ \int |s_{1}|^{p-3}s_1 \p_2 u f \phi^{2p}
+ 2p\int |s_{1}|^{p-3}s_1 f   X_1\phi \phi^{2p-1}$$$$
+ (p-2)\int |s_{1}|^{p-3}X_1 s_1 \sign(s_1) f  \phi^{2p}
\leq$$

$$\leq C_1\int
|s_1|^{p-1}\big(|\nabla_\e\phi|^{2}+\phi^2\big)\phi^{2p-2} +
\frac{C}{\delta}\int|X_2z|^{p-1} \phi^{2p}+ \delta \int|\nabla_\e s_1|^2
|s_1|^{p-3} \phi^{2p}+$$
$$+ \int|\nabla^2_\e u| |\nabla_\e z|^{p-1} \phi^{2p} 
+ \frac{C}{\delta} \int|\nabla^2_\e u|^2 |\nabla_\e z|^{p-1} \phi^{2p} + 
\int|\nabla^2_\e u| |\nabla_\e z|^{p-1} \phi^{2p-1}|\nabla_\e \phi| +  $$$$+ C\int  |X_2( \p_2 u)|^p  \phi^{2p} +  \int |\nabla_\e z|^{p} \phi^{2p}+ \frac{C}{\delta}\int |s_1| ^{p+1/2} \phi^{2p} +$$$$+ \frac{C}{\delta} \int |f| ^{(2p+1)/7} \phi^{2p} + \int |f|^{(2p+1)/5}(|\nabla_\e \phi| + \phi)^{(2p+1)/5} \phi^{(8p-1)/5}  \leq $$

 $$ \leq C_1 \int
\big(|\nabla_\e\phi|^{2}+\phi^2\big)^p 
  +  \frac{C}{\delta} \int |\nabla_\e z|^{p} \phi^{2p}+  \frac{C}{\delta} \int |s_1| ^{p+1/2} \phi^{2p} +$$$$
+ \int  |X_2( \p_2 u)|^p  \phi^{2p} +\delta\int
|s_{1}|^{p-3}|\nabla_\e s_1|^2 \phi^{2p} +  $$
$$+ \int|\nabla^2_\e u| |\nabla_\e z|^{p-1} \phi^{2p} 
+  \frac{C}{\delta}\int|\nabla^2_\e u|^2 |\nabla_\e z|^{p-1} \phi^{2p} + 
\int|\nabla^2_\e u| |\nabla_\e z|^{p-1} \phi^{2p-1}|\nabla_\e \phi| + $$
$$+  \frac{C}{\delta}\int |f| ^{(2p+1)/7} \phi^{2p} + \int |f|^{(2p+1)/5}(|\nabla_\e \phi| + \phi)^{(2p+1)/5} \phi^{(8p-1)/5} $$
  It follows
that
\begin{equation}\label{equi}\int |\nabla_\e ( |s_1| ^{(p-1)/2})|^2
\phi^{2p}\leq C \Big(\int \big(|\nabla_\e\phi|^{2}+\phi^2\big)^p
+\end{equation}$$+ \int |\nabla_\e z|^{p+1/2} \phi^{2p} 
+ \int  |X_2( \p_2 u)|^p  \phi^{2p} +  \int |f| ^{(2p+1)/7} \phi^{2p} + \int |f|^{(2p+1)/5}(|\nabla_\e \phi| + \phi)^{(2p+1)/5} \phi^{(8p-1)/5} $$ 
$$+ \int|\nabla^2_\e u| |\nabla_\e z|^{p-1} \phi^{2p} 
+ \int|\nabla^2_\e u|^2 |\nabla_\e z|^{p-1} \phi^{2p} + 
\int|\nabla^2_\e u| |\nabla_\e z|^{p-1} \phi^{2p-1}|\nabla_\e \phi|\Big) .$$
An analogous
estimate holds for $s_2=X_2z$, i.e.
\begin{equation}\label{equi2}\int |\nabla_\e ( |s_2| ^{(p-1)/2})|^2
\phi^{2p}\leq C \Big(\int \big(|\nabla_\e\phi|^{2}+\phi^2\big)^p
+\end{equation}$$
+ \int |\nabla_\e z|^{p+1/2} \phi^{2p} 
+ \int  |X_2( \p_2 u)|^p  \phi^{2p} +  \int |f| ^{(2p+1)/7} \phi^{2p} + \int |f|^{(2p+1)/5}(|\nabla_\e \phi| + \phi)^{(2p+1)/5} \phi^{(8p-1)/5} $$
$$+ \int|\nabla^2_\e u| |\nabla_\e z|^{p-1} \phi^{2p} 
+ \int|\nabla^2_\e u|^2 |\nabla_\e z|^{p-1} \phi^{2p} + 
\int|\nabla^2_\e u| |\nabla_\e z|^{p-1} \phi^{2p-1}|\nabla_\e \phi| \Big).$$
The conclusion follows  immediately.

\end{proof}

\begin{theorem}\label{teo1}
Let $p\geq 3$ be fixed and $u$ be a function satisfying the bound \eqref{MMMM= }. Let $f\in
C^\infty(\Omega)$, and  let $z$ be a smooth solution of equation
(\ref{eq111}). There exists a constant $C$, 
which depends on $p$ and the constant $M$ in (\ref{MM}) but is
independent of $\e$ and $z$ such that for every $\phi\in
C^\infty_0(\Omega),$ $\phi>0,$

$$\int |\nabla_\e z|^{p+1/2}\phi^{2p} + \int |\nabla_\e
\big(|\nabla_\e z|^{(p-1)/2}\big)|^{2}\phi^{2p} \leq$$ $$ \leq
C\Big( \int |z|^{4p+2} \phi^{2p} +  \int \big( \phi^2 + |\nabla_\e
\phi|^2\big)^p + \int |\nabla_\e( \p_2 u)|^p \phi^{2p}+ $$
$$+   \int \Big(|f|^{(2p+1)/7} 
 \phi^{2p} +   |f|^{(2p+1)/5}(|\nabla_\e \phi| + \phi)^{(2p+1)/5} \phi^{(8p-1)/5}  \Big) + \int|\nabla^2_\e u|^{2p} 
\phi^{2p}\Big). $$
\end{theorem}

\begin{proof}
 By Proposition \ref{3.3}, and Lemma \ref{lemmalteo1} calling $s_1 = X_1 z$ we have 
$$ \int |s_1| ^{p+1/2} \phi^{2p} \leq \frac{C}{\delta} \int \Big( |z|^{4p+2}
\phi^{2p} +  |z|^{(2p+1)/2} \phi^{2p}  + |z|^{(2p+1)/2} 
|\nabla_e \phi|^{(2p+1)/2}\Big) +$$$$   + \int \big(|\nabla_\e\phi|^{2}+\phi^2\big)^p +$$$$+
\delta  \int (|\nabla_\e z|^{p+1/2} + |X_2( \p_2 u)|^p ) \phi^{2p}
+ \int(|\nabla^2_\e u|^{2p} + |\nabla^2_\e u|^{p}) \phi^{2p}  + $$
$$+  C \int \Big(|f|^{(2p+1)/7} 
 \phi^{2p} +   |f|^{(2p+1)/5}(|\nabla_\e \phi| + \phi)^{(2p+1)/5} \phi^{(8p-1)/5}  \Big). $$
To estimate  $s_2= X_2z$ we argue  in the same way   and
obtain
$$ \int |s_2| ^{p+1/2}
\phi^{2p} \leq   \frac{C}{\delta} \int \Big( |z|^{4p+2}
\phi^{2p} +  |z|^{(2p+1)/2} \phi^{2p}  + |z|^{(2p+1)/2} 
|\nabla_e \phi|^{(2p+1)/2}\Big) +$$$$  + \int \big(|\nabla_\e\phi|^{2}+\phi^2\big)^p +$$$$+
\delta  \int (|\nabla_\e z|^{p+1/2} + |X_2( \p_2 u)|^p ) \phi^{2p}
+ \int(|\nabla^2_\e u|^{2p} + |\nabla^2_\e u|^{p}) \phi^{2p}  + $$
$$+  C \int \Big(|f|^{(2p+1)/7} 
 \phi^{2p} +  |f|^{(2p+1)/5}(|\nabla_\e \phi| + \phi)^{(2p+1)/5} \phi^{(8p-1)/5} \Big). $$
Hence, if $\delta$ is sufficiently small $$ \int \Big(|s_1| ^{p+1/2} +
|s_2|^{p+1/2}\Big) \phi^{2p} \leq \frac{C}{\delta} \int \Big( |z|^{4p+2}
\phi^{2p} +  |z|^{(2p+1)/2} \phi^{2p}  + |z|^{(2p+1)/2} 
|\nabla_\e \phi|^{(2p+1)/2}\Big) +$$$$+ \int \big(|\nabla_\e\phi|^{2}+\phi^2\big)^p + 
 \int |\nabla_e( \p_2 u)|^p   \phi^{2p} + \int(|\nabla^2_\e
u|^{2p} + |\nabla^2_\e u|^{p}) \phi^{2p}+ $$
$$+  C \int \Big(|f|^{(2p+1)/7} 
 \phi^{2p} +   |f|^{(2p+1)/5}(|\nabla_\e \phi| + \phi)^{(2p+1)/5} \phi^{(8p-1)/5}  \Big). $$
The conclusion follows from the latter, (\ref{equi}) and \eqref{equi2} and the H\"older inequality
$$\int|\nabla_\e u|^{p}\phi^{2p} \leq 
\int|\nabla_\e u|^{2p}\phi^{2p} + \int\phi^{2p}.
$$ 
\end{proof}

Next we  iterate once the previous result

\begin{theorem}\label{teo2}
Let $p\geq 3$ be fixed and $u$ be a function satisfying the bound \eqref{MMMM= }. Consider a  function  $f\in
C^\infty(\Omega)$, and   $z$  a smooth solution of equation
(\ref{eq111}).
 Let $\O_1,\O_2$ so that $\O_1\subset\subset \O_2\subset\subset
\Omega$. There exists a constant $C$,  which depends on
$p,$ on $\O_i$, and on the constant $M$ in (\ref{MM}), but is
independent of $\e$ or $z$ such that $$||z||^{p+1/2}_{W^{2,p+1/2}_\e(\O_1)}
+ \sum_{|I|=2} |||\nabla_\e^I z| ^{(p-1)/2} ||^2_{W^{1,2}_\e(\O_1)}\leq $$
$$ \leq C\Big( ||f||^{(2p+1)/5}_{W^{1, (2p+1)/5}_\e(\O_2)} +
 ||v||^{4p+2}_{W^{1,4p+2}_\e(\O_2)} + $$$$+  ||u||^{4p+2}_{W^{2,4p+2}_\e(\O_2)} +
 ||u||^{2p/3}_{W^{3, 2p/3}_\e(\O_2)} +
||z||^{4p+2}_{W^{1,4p+2}_\e(\O_2)} +
 ||z||^{2p/3}_{W^{2, 2p/3}_\e(\O_2)} +1\Big). $$
 Here $I$ is a multi-index and $\nabla_\e^I z$ denotes derivatives of order $|I|$ along $X_{i,u}$, $i=1,2$.
\end{theorem}

\begin{proof} If $z$  is a solution of  (\ref{eq111}) then, by
Lemma \ref{lem6} the function  $s_{2}=X_2 z$  is a solution of the
 equation:
$$M_\e s_2= \tilde f_0$$
where 
$$\tilde f_0= X_2f - X_{i}\Bigg(\frac{a_{i1}(\nabla_\e u)}{\sqrt{1 +
|\nabla_\e u|^2}} \p_2 u X_2 z \Bigg) -$$$$-
X_{i}\Bigg(X_2\Big(\frac{a_{ij}(\nabla_\e u)}{\sqrt{1 +
|\nabla_\e u|^2}}\Big) X_jz \Bigg)
 - \p_2 u  X_2\Bigg(\frac{a_{1j}(\nabla_\e u)}{\sqrt{1 +
|\nabla_\e u|^2}}X_jz \Bigg).$$
Let us choose $\O_3$ such that
$\O_1\subset\subset\O_3\subset\subset\O_2$. By Theorem \ref{teo1}
there exists a constant $C$   independent of $\e$ such
that
\begin{equation}\label{normtz}
||s_{2}||^{p+1/2}_{W^{1,p+1/2}_\e(\O_3)} +
||\ |\nabla_\e
s_{2}|^{(p-1)/2}||^2_{W^{1,2}_\e(\O_3)}\leq
\end{equation}
$$\leq  C\Big( ||\tilde f_{0}||^{(2p+1)/5}_{L^{(2p+1)/5}(\O_2)}+
||v||^{p}_{W^{1,p}_\e(\O_2)} + ||u||^{2p}_{W^{2,2p}_\e(\O_2)}  +  ||s_2||^{4p+2}_{L^{4p+2}(\O_2)}+1 \Big).$$
We note that for three fixed functions $f,g,h,$
$$ || fgh||^{(2p+1)/5}_{L^{(2p+1)/5}(\O_2)}\leq $$
$$\leq || f||^{2p/3}_{L^{2p/3}(\O_2)} +  
||g||^{4p(2p+1)/4p-3}_{L^{4p(2p+1)/4p-3}(\O_2)} +  || h||^{4p(2p+1)/4p-3}_{L^{4p(2p+1)/4p-3}(\O_2)}\leq $$
$$\leq || f||^{2p/3}_{L^{2p/3}(\O_2)} +  
||g||^{4p+2}_{L^{4p+2}(\O_2)} +  || h||^{4p+2}_{L^{4p+2}(\O_2)} + C. $$
If follows that 
$$||\tilde f_{0}||^{(2p+1)/5}_{L^{(2p+1)/5}(\O_2)}\leq $$
$$\leq ||f||^{(2p+1)/5}_{W^{1, (2p+1)/5}(\O_2)}+ 
|| \nabla_\e^2 u v \nabla_\e z||^{(2p+1)/5}_{L^{(2p+1)/5}(\O_2)}
+|| \nabla_\e v \nabla_\e z||^{(2p+1)/5}_{L^{(2p+1)/5}(\O_2)}+
$$$$+
||( \nabla_\e^2 u)^2  \nabla_\e  z||^{(2p+1)/5}_{L^{(2p+1)/5}(\O_2)}+
||(1+ \nabla_\e^2 u )v \nabla_\e^2 z||^{(2p+1)/5}_{L^{(2p+1)/5}(\O_2)}+
$$$$+ 
|| \nabla_\e^3 u \nabla_\e z||^{(2p+1)/5}_{L^{(2p+1)/5}(\O_2)}+ || \nabla_\e^2 u \nabla_\e ^2z||^{(2p+1)/5}_{L^{(2p+1)/5}(\O_2)}\leq $$ 
$$ \leq C\Big( ||f||^{(2p+1)/5}_{W^{1, (2p+1)/5}_\e(\O_2)} +
 ||v||^{4p+2}_{W^{1,4p+2}_\e(\O_2)} + $$$$+  ||u||^{4p+2}_{W^{2,4p+2}_\e(\O_2)} +
 ||u||^{2p/3}_{W^{3, 2p/3}_\e(\O_2)} +
||z||^{4p+2}_{W^{1,4p+2}_\e(\O_2)} +
 ||z||^{2p/3}_{W^{2, 2p/3}_\e(\O_2)} +1\Big). $$
Arguing in the same way with  the function $s_{1}=X_1z$ we conclude
the proof.\end{proof}

\bigskip

Iterations of Theorem \ref{teo1} yield the following

\begin{theorem}\label{teo3}
Let $p\geq 3$, $m\geq 1$ be a fixed positive integer and $u$
be a function satisfying the bound \eqref{MMMM= }. Assume that
$f\in C^\infty( \Omega)$, and let $z$ be a smooth solution of equation
(\ref{eq111}) in $\Omega$. If
$\O_1\subset\subset\O_2\subset\subset \Omega$ then there exists a
constant $C$   which depends on $p,$ $\O_i$ and on
$M$ in (\ref{MM}), but is independent of $\e$ or $z$ such that
the solution satisfies the following estimate

$$||z||^{p+1/2}_{W^{m+1,p+1/2}_\e(\O_1)}
+ \sum_{|I|=m+1} |||\nabla_\e^I z| ^{(p-1)/2} ||^2_{W^{1,2}_\e(\O_1)}\leq $$
$$ \leq C\Big( ||f||^{(2p+1)/5}_{W^{m, (2p+1)/5}_\e(\O_2)} +
 ||v||^{4p+2}_{W^{m,4p+2}_\e(\O_2)} + $$$$+  ||u||^{4p+2}_{W^{m+1 ,4p+2}_\e(\O_2)} +
 ||u||^{2p/3}_{W^{m+2, 2p/3}_\e(\O_2)} +
||z||^{4p+2}_{W^{m,4p+2}_\e(\O_2)} +
 ||z||^{2p/3}_{W^{m+1, 2p/3}_\e(\O_2)} + 1\Big). $$

\end{theorem}

\bigskip

\section{A priori estimates for the non-linear approximating PDE}

We now return to the equation $L_\e u=0.$ Let $u$ be a smooth solution satisfying 
\eqref{MMMM= }. 
In view of Proposition \ref{Calphaestimate} and Theorem \ref{p402} (i) we have the following statement: for every open set 
$\O_1\subset\subset\Omega$ 
  there exists a positive constant 
constant $C$ which depends on $\O_1$ and on
$M$ in (\ref{MM}), but is independent of $\e$ such that

\begin{equation}\label{stima0}
||u|| _{W^{2, 10/3}_\e(\O_1)} 
+ ||\p_2u|| _{W^{1, 2}_\e(\O_1)} 
+ ||u|| _{C^{1, \alpha}_E(\O_1)} \leq C.
\end{equation}

Our first step is the higher integrability of the Hessian of $u$. 
The proof rests on the estimates obtainted from the freezing technique  in
Theorem \ref{p402} and from a new 
Euclidean\footnote{rather than subelliptic } Cacciopoli inequality \eqref{CacciEucl}.

\begin{lemma}\label{unormaw2p}

Let $u$ be a smooth solution of 
$$L_\e u=0,$$
in $\Om\subset \R^2$ satifying \eqref{MMMM= } and denote $v=\p_2 u$. 
For every open set 
$\O_1\subset\subset\Omega$, for every $p\geq
1$ there exists a positive constant 
$C$ which depends on $\O_1$,  $p$, and on
$M$ in (\ref{MM}), but is independent of $\e$ such that
 $$ ||u||^p_{W^{2,p}_\e(\O_1)} +  ||\nabla_\e v||^4_{L^{4} (\Omega_1)}\leq C.$$
\end{lemma}

\begin{proof}

In view of Lemma \ref{equ2n} 
the function $v=\p_{2}u$ satisfies the equation: 

$$X_{i}  \Big( \frac{a_{ij}(\nabla_\e u)}{\sqrt{1 +|\nabla_\e u|^2}}
X_{j}v \Big) =f,$$
with 
$$f = -\frac{a_{11}(\nabla_\e
u)}{\sqrt{1 +|\nabla_{\e} u|^2}} v^3 - 3 \frac{a_{1j}(\nabla_\e
u)}{\sqrt{1 +|\nabla_{\e} u|^2}}v X_j v - X_{i} \Big(\frac{a_{i1}(\nabla_\e
u)}{\sqrt{1 +|\nabla_{\e} u|^2}} \Big) v^2$$

Hence, applying Lemma \ref{lemmalteo1} one has

\begin{equation}\int |\nabla_\e ( |\nabla_\e v| ^{(p-1)/2})|^2
\phi^{2p}\leq C \Big(\int \big(|\nabla_\e\phi|^{2}+\phi^2\big)^p
+\end{equation}$$+
 \int (|\nabla_\e v|^{p+1/2}) \phi^{2p} 
+ \int  |X_2( \p_2 u)|^p  \phi^{2p} + \int |f| ^{(2p+1)/7} \phi^{2p} +$$
$$+\int |f|^{(2p+1)/5}(|\nabla_\e \phi| + \phi)^{(2p+1)/5} \phi^{(8p-1)/5} 
+$$
$$+ \int|\nabla^2_\e u| |\nabla_\e v|^{p-1} \phi^{2p} 
+ \int|\nabla^2_\e u|^2 |\nabla_\e v|^{p-1} \phi^{2p} + 
\int|\nabla^2_\e u| |\nabla_\e v|^{p-1} \phi^{2p-1}|\nabla_\e \phi|\Big) .$$

There exist positive constants $C_1 = C_1 (|\nabla_\e \phi|, \phi, M) $ and $C_2=C_2(M)$ such that for $p=3$ we obtain

\begin{equation}\int |\nabla^2_\e v|^2
\phi^{6}\leq C_1 
+
C_2\Big( \int |\nabla_\e v|^{3+1/2} \phi^{6} 
+  \int (1 + |\nabla_\e v| + |\nabla_\e^2 u|)^{7/5}\phi^{23/5} (|\nabla_{\e}\phi| + \phi)^{7/5} +\end{equation}
$$+ \int|\nabla^2_\e u| |\nabla_\e v|^{2} \phi^{6} 
+ \int|\nabla^2_\e u|^2 |\nabla_\e v|^{2} \phi^{6} + 
\int|\nabla^2_\e u| |\nabla_\e v|^{2} \phi^{5}|\nabla_\e \phi|\Big) .$$

It follows that 

\begin{equation}\label{caccinonlin1}\int |\nabla^2_\e v|^2\phi^{6} \leq \frac{C_2}{\delta}\int |\nabla_\e^2u|^4 \phi^{6} + \delta \int |\nabla_\e v|^4 \phi^{6} + \frac{C_1}{\delta}.\end{equation}

Analogously, if we set $z= X_1u$, or $z= X_2u$, 
using Lemma \ref{equXz} and arguing as above 
we have
 \begin{equation}\label{caccinonlin2}\int |\nabla_\e^2 z|^2\phi^6\leq
\frac{C_2}{\delta}\int |\nabla_\e^2u|^4 \phi^6 + \frac{C_1}{\delta} + C_2\int |\nabla_\e v|^3 \phi^6 \end{equation}

Using Lemma \ref{interpinfinity},   \eqref{caccinonlin1} and \eqref{stima0}, we obtain immediately 
$$\int |\nabla_\e v|^4 \phi^6 \leq C_1 +C_2\int |\nabla_\e^2 v|^2 \phi^6 \leq C_1 + \frac{C_2}{\delta}\int |\nabla_\e^2u|^4 \phi^6 + \delta \int |\nabla_\e v|^4 \phi^6 $$
Hence \begin{equation}\label{questa}\int |\nabla_\e v|^4 \phi^6 \leq  
C_1 + C_2\int |\nabla_\e^2u|^4\phi ^6\end{equation}

Consequently, from the latter and \eqref{caccinonlin2} we deduce that  \begin{equation}\label{questa1}\int |\nabla_\e^2 z|^4 \phi^6 \leq   
C_1 + C_2\int |\nabla_\e^2u|^4\phi ^6\end{equation}

Next, from the intrinsic Cacciopoli inequalities \eqref{questa} and \eqref{questa1} we deduce an Euclidean Cacciopoli inequality: 
Note that 
$$|\nabla_E X_1 z|\leq |X_1^2 z | + C_2 |\p_2 X_1 z| 
\leq |X_1^2 z | + C_2 |v  \p_2 z| +  C_2| X_1 \p_2  z|\leq $$
(since $\p_2 z= \p_2 X_1 u = v^2 + X_1 v$)
$$|\nabla_\e ^2 z | + C_2 |\nabla_\e ^2 v| +  C_2|\nabla_\e v| + C_2. $$
From the latter and \eqref{questa} and \eqref{questa1} we infer 
\begin{equation}\label{CacciEucl}
\int |\nabla_E \nabla_\e z|^2 \phi^6 \leq C_2 \Big(\int |\nabla_\e^2 v|^2 \phi^6  + \int |\nabla_\e^2 z|^2 \phi^6 +1\Big) \leq C_2\int |\nabla_\e z|^4 \phi^6 + C_1\end{equation}
Now we can apply the standard Euclidean Sobolev inequality in $\R^2$ and obtain

$$\Big(\int (|\nabla_\e z|\phi^3)^6\Big)^{1/3}
\leq C_2\int |\nabla_E(\nabla_\e z\phi^3)|^2 \leq  C_2\int |\nabla_\e z|^4 \phi^6 +  C_1 \leq$$
(using H\"older inequality )
$$\leq C_2\Big(\int (|\nabla_\e z|\phi^3)^6\Big)^{1/3}\Big(\int_{supp(\phi)} |\nabla_\e z|^3\Big)^{2/3} + C_1.$$
By \eqref{stima0} and the fact that $|\nabla_\e z | \leq |\nabla_\e^2 u|,$ we already know that 
$|\nabla_\e z|\in L^3_{loc}.$ 
In fact $$\Big(\int_{supp(\phi)} |\nabla_\e z|^3\Big)^{2/3} \leq \Big(\int_{supp(\phi)} |\nabla_\e z|^{10/3}\Big)^{3/5} |supp(\phi)|^{1/15}.$$ Recall that $C_2$ doen not depend on $|\nabla_\e \phi|$. 
If we choose the support of $\phi$ sufficiently small, we can assume that the integral  
$\int_{supp(\phi)} |\nabla_\e  z|^3$ is arbitrarily small. 
It follows that 
$$\Big(\int (|\nabla_\e z|\phi^3)^6\Big)^{1/3}\leq C_1$$
and consequently, by (\ref{questa})
$$\int |\nabla_\e v|^4\phi^6 \leq C_1$$
But this implies that
$|\nabla_E(\nabla_\e u)|\leq |\nabla_\e^2 u| + |\nabla_\e v| + v^2 \in L^4_{loc}$. 
This implies, buy the standard Euclidean Sobolev Morrey inequality in $\R^2$ that 
$$\nabla_\e u\in C^{1/2}_E.$$

By Theorem \ref{p402} (ii)
it then follows that for every $r>1$ there exists a constant $C>0$ independent of $\e$ such that 
$$||\nabla_\e ^2 u||_{W^{2, r}}\leq C_1.$$

\end{proof}

In order to boostrap regularity we apply Theorem \ref{teo3} to the non linear equation $L_\e u=0,$
and obtain immediately the following: 

\begin{lemma}\label{5.2primo}
Let $u$ be a smooth solution of $L_\e u =0$, satisfying \eqref{MMMM= }. Set $z=X_iu$, $i=1,2$ 
$v=\p_2 u$. For every open set  $\O_1\subset\subset\O_2\subset\subset \Omega$, 
for every $p\geq 3$, and every integer $m\geq 2$  there exist a constant $C$  which depend on $p,m$  $\O_i$ and on 
$M$ in (\ref{MM}), but is independent of $\e$ such that
 the following estimates
hold

\begin{equation}\label{iterationa}
||z||^{p+1/2}_{W^{m,p+1/2}_\e(\O_1)}+
||v||^{p+1/2}_{W^{m,p+1/2}_\e(\O_1)}
\leq \end{equation}
$$ \leq C\Big( 
 ||v||^{4p+2}_{W^{m-1,4p+2}_\e(\O_2)} +  ||z||^{4p+2}_{W^{m-1,4p+2}_\e(\O_2)} +
  ||v||^{2p/3}_{W^{m, 2p/3}_\e(\O_2)}
+ ||z||^{2p/3}_{W^{m, 2p/3}_\e(\O_2)} +1\Big). $$

\begin{equation}\label{iterationb}
||z||^{2}_{W^{m+1,2}_\e(\O_1)}+
||v||^{2}_{W^{m+1,2}_\e(\O_1)}
\leq \end{equation}
$$ \leq C\Big( 
 ||v||^{14}_{W^{m-1,14}_\e(\O_2)} +  ||z||^{14}_{W^{m-1,14}_\e(\O_2)} +
  ||v||^{2}_{W^{m, 2}_\e(\O_2)}
+ ||z||^{2}_{W^{m, 2}_\e(\O_2)} +1\Big). $$

\end{lemma}

\begin{proof}

In view of Lemma \ref{equ2n}
the function $v$ solves an equation of the form $$M_\e v = f_v,$$
with
\begin{equation} \label{effev}f_v = -\frac{a_{11}(\nabla_\e
u)}{\sqrt{1 +|\nabla_{\e} u|^2}} v^3 - 3 \frac{a_{1j}(\nabla_\e
u)}{\sqrt{1 +|\nabla_{\e} u|^2}}v X_j v - X_{i} \Big(\frac{a_{i1}(\nabla_\e
u)}{\sqrt{1 +|\nabla_{\e} u|^2}} \Big) v^2.\end{equation}

Analogously, the function $z= X_i u$ solves the equation $$M_\e z = f_z,$$ of the form 
with
\begin{equation} \label{effez}f_z =-[X_{k}, X_{i}] \Big(\frac{X_{i}u}{\sqrt{1 +
|\nabla_{\e} u |^2}}\Big) - X_{i} \Big(\frac{a_{ij}(\nabla_\e
u)}{\sqrt{1 +|\nabla_{\e} u|^2}} [X_{k}, X_{j}] u\Big).\end{equation}

Hence 
\begin{equation}\label{fvfz}|f_v| + |f_z| \leq C (1+  |\nabla_\e v| + \nabla_\e z|),
\end{equation}
for some constant $C$ depending only on $M$ in \eqref{MM}.

Applying Theorem \ref{teo3} to $z$ and to $v$, 
yields:

$$||z||^{p+1/2}_{W^{m,p+1/2}_\e(\O_1)}
+ ||v||^{p+1/2}_{W^{m,p+1/2}_\e(\O_1)}+
\sum_{|I|=m} |||\nabla_\e^I z| ^{(p-1)/2} ||^2_{W^{1,2}_\e(\O_1)}+
\sum_{|I|=m} |||\nabla_\e^I v| ^{(p-1)/2} ||^2_{W^{1,2}_\e(\O_1)}\leq $$
$$ \leq C\Big( ||f_z||^{(2p+1)/5}_{W^{m-1, (2p+1)/5}_\e(\O_2)} +
||f_v||^{(2p+1)/5}_{W^{m-1, (2p+1)/5}_\e(\O_2)} +
$$$$+
 ||v||^{4p+2}_{W^{m-1,4p+2}_\e(\O_2)} +  ||u||^{4p+2}_{W^{m ,4p+2}_\e(\O_2)} +
 ||u||^{2p/3}_{W^{m+1, 2p/3}_\e(\O_2)} + $$$$+ 
||z||^{4p+2}_{W^{m-1,4p+2}_\e(\O_2)} +
 ||z||^{2p/3}_{W^{m, 2p/3}_\e(\O_2)} + 
||v||^{4p+2}_{W^{m-1,4p+2}_\e(\O_2)} +
 ||v||^{2p/3}_{W^{m, 2p/3}_\e(\O_2)} + 1\Big)\leq $$
(substituting \eqref{fvfz} in the latter)
$$ \leq C\Big( ||\nabla_\e z||^{(2p+1)/5}_{W^{m-1, (2p+1)/5}_\e(\O_2)} +
||\nabla_\e v||^{(2p+1)/5}_{W^{m-1, (2p+1)/5}_\e(\O_2)} +
$$$$+
||z||^{4p+2}_{W^{m-1,4p+2}_\e(\O_2)} +
 ||z||^{2p/3}_{W^{m, 2p/3}_\e(\O_2)} + 
||v||^{4p+2}_{W^{m-1,4p+2}_\e(\O_2)} +
 ||v||^{2p/3}_{W^{m, 2p/3}_\e(\O_2)} + 1\Big)\leq $$
(using H\"older inequality)
$$\leq C\Big(
||z||^{4p+2}_{W^{m-1,4p+2}_\e(\O_2)} +
 ||z||^{2p/3}_{W^{m, 2p/3}_\e(\O_2)} + 
||v||^{4p+2}_{W^{m-1,4p+2}_\e(\O_2)} +
 ||v||^{2p/3}_{W^{m, 2p/3}_\e(\O_2)} + 1\Big). $$

We have proved \eqref{iterationa} and 
substituting $p=3$ yields  \eqref{iterationb} 

\end{proof}

The main result of this section is the following a priory regularity estimates for solutions of the approximating non linear equation:

\begin{theorem}\label{prop32}
Let $u$ be a smooth solution of 
$$L_\e u=0,$$ 
in $\Om\subset \R^2$,  satisfying \eqref{MMMM= }. For every open set  $\O_1\subset\subset \Omega$, 
for every $p\geq 3$, and every integer $m\geq 2$  there exist a constant $C$  which depends on $p,m$  $\O_1$ and on 
$M$ in (\ref{MM}), but is independent of $\e$ such that the following estimates holds

\begin{equation}\label{star}
||u||_{W^{m,p}_\e(\O_1)}+  
||\p_2 u||_{W^{m,p}_\e(\O_1)} \leq
C.\end{equation}

\end{theorem}

\begin{proof}

The proof follows from the estimate 
\begin{equation}\label{inductive}
||X_1 u||_{W^{m-1,p}_\e(\O_1)}+  
||\p_2 u||_{W^{m-1,p}_\e(\O_1)} +
||X_1 u||_{W^{m,2}_\e(\O_1)}+  
||\p_2 u||_{W^{m,2}_\e(\O_1)} 
\leq
C,\end{equation}
which we prove by induction.

\medskip 

\noindent {\bf First step: $m=2$}

By Lemma \ref{unormaw2p} we already know that there exists a constant 
such that for every $p$ 
\begin{equation}\label{2derivu}
 ||u||^p_{W^{2,p}_\e(\Omega_1)} \leq C_1.\end{equation}
 We need to show that $v\in W^{1,p}_{\e loc}$ for every $p$, and that $X_1 u, v \in W^{2,2}_{\e loc}$. 
 
 Note that we can not yet invoke Lemma \ref{5.2primo}, since it only apply to higher order derivatives. 
 
Recall that $v$ is a solution of $M_\e v = f_v,$
where $f_v$ is defined in \eqref{effev}
By Theorem \ref{teo1} there exist a constant $C$ 
such that 

$$\int |\nabla_\e v|^{p+1/2}\phi^{2p} + \int |\nabla_\e
\big(|\nabla_\e v|^{(p-1)/2}\big)|^{2}\phi^{2p}\leq
$$$$ \leq
C\Big(\int |v|^{4p+2} \phi^{2p} +  \int \big( \phi^2 + |\nabla_\e
\phi|^2\big)^p + \int |\nabla_\e( \p_2 u)|^p \phi^{2p}+ $$
$$+  \int \big(|f_v|^{(2p+1)/7} 
 \phi^{2p} +  |f_v|^{(2p+1)/5} 
 (|\nabla_\e\phi|+\phi)^{(2p+1)/5} \phi^{(8p-1)/5}\big) + \int|\nabla^2_\e u|^{2p} 
\phi^{2p}\Big)\leq $$
(using \eqref{fvfz}, \eqref{stima0} and \eqref{2derivu} )
$$\leq 
C\Big( 1 + \int |\nabla_\e v|^{(2p+1)/5} (|\nabla_\e\phi|+\phi)^{(2p+1)/5} \phi^{(8p-1)/5} + \int |\nabla_\e( \p_2 u)|^p \phi^{2p}\Big) \leq $$$$\leq
C\Big( 1 + \int  |\nabla_\e v|^p (|\nabla_\e\phi|+\phi)^{(2p+1)/5} \phi^{(8p-1)/5} \Big). $$
From Lemma \ref{unormaw2p} the right hand side 
is bounded for $p=4.$ An obvious boostrap argument yields $\nabla_\e v\in L^p_{loc}$ for every $p$. 
Moreover, choosing $p=3$, we also infer $$\nabla_\e^2 v\in L^2_{loc}.$$

\smallskip

To conclude the first iteration step we observe that 
the function $z= X_1 u$ solves $M_\e z = f_z,$ where
$f_z$ is defined in \eqref{effez}. Theorem \ref{teo1}
and estimate \eqref{fvfz} yield that there exists a constant $C$ 
such that 
$$\int |\nabla_\e z|^{p+1/2}\phi^{2p} + \int |\nabla_\e
\big(|\nabla_\e z|^{(p-1)/2}\big)|^{2}\phi^{2p} \leq C.$$ 
Choosing $p=3$ we obtain $X_1 u \in W^{2,2}_{\e loc}$

\medskip 

\noindent {\bf Main iteration  step: $m>2$}

Assume \eqref{inductive} holds for for a fixed value of $m$.

  Let $\Omega_2$ be as in Lemma \ref{5.2primo}. In view of that result we infer

\begin{equation} ||z||^{p+1/2}_{W^{m,p+1/2}_\e(\O_1)}+
||v||^{p+1/2}_{W^{m,p+1/2}_\e(\O_1)}
\leq \end{equation}
$$ \leq C\Big( 
 ||v||^{4p+2}_{W^{m-1,4p+2}_\e(\O_2)} +  ||z||^{4p+2}_{W^{m-1,4p+2}_\e(\O_2)} +
  ||v||^{2p/3}_{W^{m, 2p/3}_\e(\O_2)}
+ ||z||^{2p/3}_{W^{m, 2p/3}_\e(\O_2)} +1\Big)\leq$$
by induction assumption
$$\leq C\Big( 
  ||v||^{2p/3}_{W^{m, 2p/3}_\e(\O_2)}
+ ||z||^{2p/3}_{W^{m, 2p/3}_\e(\O_2)} +1\Big).$$

The same boostrap argument used above implies $v, z \in W^{m,p}_{\e loc}$ for every $p$. 

Invoking \eqref{iterationb}
\begin{equation}
||z||^{2}_{W^{m+1,2}_\e(\O_1)}+
||v||^{2}_{W^{m+1,2}_\e(\O_1)}
\leq \end{equation}
$$ \leq C\Big( 
 ||v||^{14}_{W^{m-1,14}_\e(\O_2)} +  ||z||^{14}_{W^{m-1,14}_\e(\O_2)} +
  ||v||^{2}_{W^{m, 2}_\e(\O_2)}
+ ||z||^{2}_{W^{m, 2}_\e(\O_2)} +1\Big)\leq C,$$
  concluding the proof. 
\end{proof}

%
%
%

\section{Estimates for the viscosity solution}
In this section we turn our attention to the proof of regularity
for  vanishing  viscosity solutions $u$
of equation (\ref{eq0bis}). The regularity is expressed in terms of the intrinsic  Sobolev spaces
 $W^{k,p}_0(\O)$ and rests on   the {\it a priori} estimates proved
 in the previous section in the limit $\e\to 0$.

Let $u$ be a vanishing  viscosity solution, and ($u_j$) denote its
approximating sequence, as defined in Definition \ref{visc}. For
each $\e_j$ and  function $u_j$ we set $X_{1,j}=\p_1+u_j\p_{x_2}$, 
$X_{2,j}=\e_j \p_{x_2}$ the
corresponding vector fields, and let $\nabla_{\e_j}$ and $W^{k,p}_{\e_j}(\O)$ denote
the natural gradient and Sobolev spaces. We also let $u$ , $X_1=\p_1+u\p_2$, and $\nabla_0=(X_1,0)$
denote the coefficients and vector fields associated to the limit
equation  and the limit solution $u$, while $W^{k, p}_0(\O)$ will be the associated
Sobolev space. Note that $\nabla_E$ and $W^{k,p}_E(\O)$ are the usual
gradient and Sobolev space.
\bigskip

\begin{theorem}\label{CiMo} Let $u \in Lip(\Omega)$
be a vanishing  viscosity solution of {\rm(\ref{eq0bis})}, and set $v_j =
\p_2u_j$. For every ball $B(R) \subset \subset \O$ and $p>1$ there
exists a constant $C>0$ such that
\begin{equation}\label{eqde} ||\nabla_{\e_j} u_j||_{W^{1,p}_E(B(R))} +
|| v_j||_{L^\infty(B(R))}+ ||v_j ||_{W^{1,2}_{\e_j}(B(R))} \leq
C\end{equation} and
\begin{equation}\label{xj}X_{1,j} u_j\rightarrow Xu, \quad X_{2,j} u_j
 \rightarrow 0
\end{equation} as $j \rightarrow +\infty$ weakly in $W^{1,2}_{E, \loc}(\Omega)$. Moreover
equation (\ref{eq0bis}) can be represented as $$X^2u =0$$ and is satisfied
 weakly in the Sobolev sense, and hence, pointwise a.e. in $ \O$, i.e.
$$\int_\O XuX^*\phi =0 \text{ for all }\phi\in C^{\infty}_0(\O).$$
 \end{theorem}
 \begin{proof} The uniform bound on $||v_j||_{L^\infty(B(R))}$
 follows from the definition of vanishing viscosity solution.
 The bound on $||v_j ||_{W^{1,2}_{\e_j}(B(R))}$ is a consequence
 of \eqref{star}.  To prove the remaining estimate observe that
 for any function $w$: $\p_2 X_{1,j} w=X_{1,j} \p_2 w+\p_2 u_j \p_2 w$.
 Substituting $w=u_j$ and
in view of \eqref{star} 
we see that there exists positive constants
$C_1,C_2$ depending only on the uniform bound on $||v_j||_{L^\infty(B(R))}$ such that for any $p\ge 1$,
\begin{multline}
||\p_1 \nabla_{\e_j} u_j||_{L^p(B(R)}+||\p_2 \nabla_{\e_j} u_j||_{L^p(B(R)}\\
\le ||X_{1,j} \nabla_{\e_j}u_j||_{L^p(B(R)}+||(1+|u_j|)\p_2 \nabla_{\e_j} u_j||_{L^p(B(R)}
\\
\le
 ||u_j||_{W^{2,p}_{\e}(B(R))} +C_1||v_j||_{W^{1,p}_{\e}(B(R))}+C_2\le  C,
\end{multline} for a new constant $C>0$ independent of $j$.
 The weak regularity of
$u$ and the weak Sobolev convergence follow in a standard fashion. 

Next we address the PDE:
Since for every $j$ the approximating solution $u_j$ is of class $C^\infty$ then we can  use the non divergence form of the equation 
$$
\sum_{h,k=1}^{2} a_{h k}(\nabla_j u_j) X_{h, j}X_{k,
j}u_j=0.$$
Here 
$$a_{h,k}(\nabla_j u_j) \rightarrow a_{h,k}(\nabla_0 u) = \delta_{h 1} \delta_{k1} \text{  in  } L^p,$$
while 

\begin{equation}X_{1,j} u_j\rightarrow Xu, \quad X_{2,j} u_j
 \rightarrow 0
\end{equation} as $j \rightarrow +\infty$ weakly in $W^{1,2}_{loc}(\Omega)$. 
Hence letting $j$ go to $\infty$ in the non divergence form equation we conclude
 $$X^2u =0$$ in the Sobolev sense. 
 \end{proof}

\bigskip

An analogous result holds for  higher order derivatives:

\begin{proposition}\label{prop41}
For every $k\in N$ for every $p>1$ and for every multiindex $I$ of
length $k$, the sequence $(\nabla_{\e_j}^I u_j)$ is bounded in
$W^{1,p}_{E,\loc}(\Omega)$. Moreover
$$X_{1,j}^k u_j \rightarrow X^k u, \text{ and }X_{2,j}^k u_j\rightarrow 0$$
weakly in $W^{1,p}_E(\O)$ as $j\to \infty$. We will express this convergence in the notation
$$ \nabla_{\e_j}^I u_j \rightarrow
D_0^I u \;\; as \;\; j \rightarrow +\infty, \;\; weakly \;\;in
\;\; W^{1,p}_E(\O).$$
\end{proposition}
\begin{proof}
Arguing as in the previous result and using \eqref{star}, and (\ref{eqde})  we deduce that for
every ball $B(R)\subset\subset \O$ there exist $C_1, C_2,\tilde C$ independent of $j$, such
that $$ ||\p_2 D^I_{\e_j}u_j||_{L^p(B(R))} \leq
C_1\sum_{|J|\le|I|} ||D^J_{\e_j}v_j||_{L^p(B(R))} + C_2\le  \tilde C,$$ $$||
\p_xD^I_{\e_j}u_j||_{L^p(B(R))} \leq ||
X_{1,j} D^I_{\e_j}u_j||_{L^p(B(R))} + || u_j
\p_2D^I_{\e_j}u_j||_{L^p(B(R))}\leq \tilde C,$$
so
 that $(\nabla_{\e_j}^I
u_j)$ is bounded in $W^{1,p}_{E,\loc}(\Omega)$ for every $p>1$, and
every multi-index $I$. Let us prove that the weak limit is $D_0^I
u$. If $|I|=1$ the assertion is true by (\ref{xj}). If $I$ is a
multi-index such that $|I|=k$, we can assume by simplicity that $I
= (1, I')$, where $|I'|=k-1$. We can also assume by inductive
hypothesis that $$\p_2 u_j \rightarrow \p_2 u \;\; as \;\; j
\rightarrow \infty \;\; weakly \;\; in \;\; L^p(\O)$$ $$u_j
\rightarrow u \;\; as \;\; j \rightarrow \infty \;\; in \;\;
L^p_{loc}(\O)$$ $$\nabla_{\e_j}^{I'}u_j \rightarrow D_0^{I'} u \;\; as
\;\; j \rightarrow \infty \;\;in \;\; L^p_{loc}(\O).$$ Then
integrating by parts $$\lim_{j\rightarrow \infty}\int
\nabla_{\e_j}^{I}u_j \phi = - \lim_{j\rightarrow \infty} \int
\nabla_{\e_j}^{I'}u_j X_j\phi - \int \p_2 u_j \nabla_{\e_j}^{I'}u_j \phi = $$
$$= - \int D_0^{I'}u X\phi - \int \p_2 u D_0^{I'}u \phi,$$ and this
ensures the weak convergence of $(D^I_{\e_j}u_j)$ to $D^I_{0}u.$
\end{proof}

\begin{remark} In view of the Ascoli-Arzel\'a  theorem and the Morrey-Sobolev embedding one has convergence
$X_{1,j}^k u_j\to X^k u$ in
the $C^{\alpha}$ norm on compact subsets of $\O$ for all
$\alpha\in (0,1)$.
\end{remark}

We can now prove the main regularity properties of the limit
function $u$:

\begin{proposition}\label{prop43}
For every $k$, and  for every $p>1$  the function  $z= X^k u $ belongs to $
 W^{1,p}_{E,\loc}(\O) $ and it is an a.e. solution
of
\begin{equation} \label{eqz}X^2z  = 0 \quad in \;\;\O.\end{equation} In
particular
\begin{equation}\label{hold}
 X^k u \in C^\alpha_{loc}(\O)
 \end{equation}
for every $0<\alpha<1$.
\end{proposition}
\begin{proof}Since $u$ is a vanishing viscosity  solution of $X^2 u =0$ in $\O$,  then Proposition \ref{prop41} implies $X^2 u \in W^{1,p}_{E,\loc}(\O)$ for all $p\ge 1$. As $X^2u=0$ a.e. in $\O$, then a simple iteration
shows that
all the derivatives $X^2 X^k u$ vanish a.e.
 in $\O$.  The H\"older regularity  \eqref{hold} follows
 from  the classical Morrey-Sobolev embedding
theorem.
\end{proof}

\medskip

We can now give a new pointwise definition of derivative in the
direction of vector fields $X_1$  and $X_2$.

\begin{definition}\label{deriv}
Let $X$ be a Lipschitz vector field on $\Om$ and let  $\xi_0\in \Omega$ and $\gamma$ be a solution to
problem $\gamma' = X(\gamma)$, $\gamma(0)=\xi_0$.

We say that a function $f\in C^{\alpha}_{loc}(\O),$ with $\alpha
\in ]0,1[,$ has Lie-derivative in the direction of the vector
field $X$ in $\xi_0$ if there exists $$ \frac{d}{dh}(f\circ\gamma)_{|h=0},$$ and we will denote its value by
$Xf(\xi_0).$
\end{definition}

\bigskip

If the weak derivative of a function $f$ is sufficiently regular,
then the two notions of derivatives coincide. For the proof of the following
result see \cite[Remark 5.6]{CLM}.

\begin{proposition} \label{lieweak}
If $f\in C_{loc}^{\alpha}(\Omega)$ for some $\alpha\in ]0,1[$ and
its weak derivatives $Xf\in C_{loc}^{\alpha}(\Omega),
\partial_2f\in L^p_{loc}(\Omega)$ with $p>1/\alpha,$ then for all
$\xi \in \Omega$ the Lie-derivatives $Xf(\xi)$ exist and coincide
with the weak ones.
\end{proposition}

We are now ready to prove the result concerning the foliation

\noindent {\bf Proof of Corollary \ref{teo01}} The equation $\gamma'=
X(\gamma) $ has an unique solution, of the form
$$\gamma(x) = (x, y(x)),$$
where $y'(x) = u(x, y(x)).$ In view of the regularity of $u$ and
of the previous proposition then
 $y''(x) = Xu(x, y(x)),$ and
$y'''(x) =X^2 u (x,y(x))=0.$
This shows that  $\gamma$ is a polynomial of order 2 and concludes
the proof. \qed
\bibliographystyle{acm}
\bibliography{GiLu}

\begin{thebibliography}{10}

\bibitem{ascv}
{\sc Ambrosio, L., Serra~Cassano, F., and Vittone, D.}
\newblock Intrinsic regular hypersurfaces in {H}eisenberg groups.
\newblock {\em J. Geom. Anal. 16}, 2 (2006), 187--232.

\bibitem{bascv}
{\sc Barone~Adesi, V., Serra~Cassano, F., and Vittone, D.}
\newblock The {B}ernstein problem for intrinsic graphs in {H}eisenberg groups
  and calibrations.
\newblock {\em Calc. Var. Partial Differential Equations 30}, 1 (2007), 17--49.

\bibitem{fsc-bigolin}
{\sc Bigolin, F., and Serra~Cassano, F.}
\newblock Intrinsic regular graphs in heisenberg groups vs. weak solutions of
  non linear first-order pdes.
\newblock Preprint, 2007.

\bibitem{BLU}
{\sc Bonfiglioli, A., Lanconelli, E., and Uguzzoni, F.}
\newblock Fundamental solutions for non-divergence form operators on stratified
  groups.
\newblock {\em Trans. Amer. Math. Soc. 356}, 7 (2004), 2709--2737.

\bibitem{CCM2}
{\sc Capogna, L., Citti, G., and Manfredini, M.}
\newblock Smoothness of lipschitz minimal intrinsic graphs in heisenberg groups
  $\mathbb{H}^n$, $n>1$.
\newblock {\em preprint\/} (2007).

\bibitem{cdg:isoperimetric}
{\sc Capogna, L., Danielli, D., and Garofalo, N.}
\newblock An isoperimetric inequality and the {S}obolev embedding theorem for
  vector fields.
\newblock {\em Math.\ Res.\ Lett. 1\/} (1994), 263--268.

\bibitem{cdpt:survey}
{\sc Capogna, L., Danielli, D., Pauls, S., and Tyson, J.}
\newblock {\em An introduction to the {H}eisenberg group and the
  sub-{R}iemannian isoperimetric problem}, vol.~259 of {\em Progress in
  Mathematics}.
\newblock Birkh\"auser Verlag, Basel, 2007.

\bibitem{ch:minimal}
{\sc Cheng, J.-H., and Hwang, J.-F.}
\newblock Properly embedded and immersed minimal surfaces in the {H}eisenberg
  group.
\newblock {\em Bull.\ Austral.\ Math.\ Soc. 70}, 3 (2004), 507--520.

\bibitem{chmy:minimal}
{\sc Cheng, J.-H., Hwang, J.-F., Malchiodi, A., and Yang, P.}
\newblock Minimal surfaces in pseudohermitian geometry.
\newblock {\em Ann.\ Sc.\ Norm.\ Super.\ Pisa Cl.\ Sci.\ (5) 4}, 1 (2005),
  129--177.

\bibitem{chy:regular}
{\sc Cheng, J.-H., Hwang, J.-F., and Yang, P.}
\newblock Regularity of $c^1$ smooth surfaces with prescribed $p-$mean
  curvature in the heisenberg group.
\newblock preprint, 2007.

\bibitem{chy}
{\sc Cheng, J.-H., Hwang, J.-F., and Yang, P.}
\newblock Existence and uniqueness for {$p$}-area minimizers in the
  {H}eisenberg group.
\newblock {\em Math. Ann. 337}, 2 (2007), 253--293.

\bibitem{CGL}
{\sc Citti, G., Garofalo, N., and Lanconelli, E.}
\newblock Harnack's inequality for sum of squares of vector fields plus a
  potential.
\newblock {\em Amer. J. Math. 115}, 3 (1993), 699--734.

\bibitem{CLM}
{\sc Citti, G., Lanconelli, E., and Montanari, A.}
\newblock Smoothness of {L}ipchitz-continuous graphs with nonvanishing {L}evi
  curvature.
\newblock {\em Acta Math. 188}, 1 (2002), 87--128.

\bibitem{CittiManfredini}
{\sc Citti, G., and Manfredini, M.}
\newblock Implicit function theorem in {C}arnot-{C}arath\'eodory spaces.
\newblock {\em Commun. Contemp. Math. 8}, 5 (2006), 657--680.

\bibitem{CiMa-F}
{\sc Citti, G., and Manfredini, M.}
\newblock Uniform estimates of the fundamental solution for a family of
  hypoelliptic operators.
\newblock {\em Potential Anal. 25}, 2 (2006), 147--164.

\bibitem{CM}
{\sc Citti, G., and Montanari, A.}
\newblock Analytic estimates for solutions of the {L}evi equations.
\newblock {\em J. Differential Equations 173}, 2 (2001), 356--389.

\bibitem{CPP}
{\sc Citti, G., Pascucci, A., and Polidoro, S.}
\newblock On the regularity of solutions to a nonlinear ultraparabolic equation
  arising in mathematical finance.
\newblock {\em Differential Integral Equations 14}, 6 (2001), 701--738.

\bibitem{CS}
{\sc Citti, G., and Sarti, A.}
\newblock A cortical based model of perceptual completion in the
  roto-translation space.
\newblock {\em J. Math. Imaging Vision 24}, 3 (2006), 307--326.

\bibitem{CittiTomassini}
{\sc Citti, G., and Tomassini, G.}
\newblock Levi equation for almost complex structures.
\newblock {\em Rev. Mat. Iberoamericana 20}, 1 (2004), 151--182.

\bibitem{dgn:minimal}
{\sc Danielli, D., Garofalo, N., and Nhieu, D.-M.}
\newblock Sub-{R}iemannian calculus on hypersurfaces in {C}arnot groups.
\newblock {\em Adv. Math. 215}, 1 (2007), 292--378.

\bibitem{dgnp1}
{\sc Danielli, D., Garofalo, N., Nhieu, D.~M., and Pauls, S.~D.}
\newblock Instability of graphical strips and a positive answer to the
  bernstein problem in the heisenberg group.
\newblock {\em to appear in Jour. Diff. Geom.\/}.

\bibitem{dgn2}
{\sc Danielli, D., Garofalo, N., and Nhieu, D.-N.}
\newblock A notable family of entire intrinsic minimal graphs in the heisenberg
  group which are not perimeter minimizing.
\newblock preprint 2006.

\bibitem{fol:1975}
{\sc Folland, G.~B.}
\newblock Subelliptic estimates and function spaces on nilpotent {L}ie groups.
\newblock {\em Ark. Mat. 2}, 13 (1975), 161--207.

\bibitem{frss:perimeter}
{\sc Franchi, B., Serapioni, R., and Serra-Cassano, F.}
\newblock Rectifiability and perimeter in the {H}eisenberg group.
\newblock {\em Math.\ Ann. 321}, 3 (2001), 479--531.

\bibitem{frss:carnot}
{\sc Franchi, B., Serapioni, R., and Serra~Cassano, F.}
\newblock Regular hypersurfaces, intrinsic perimeter and implicit function
  theorem in {C}arnot groups.
\newblock {\em Comm.\ Anal.\ Geom. 11}, 5 (2003), 909--944.

\bibitem{gn:isoperimetric}
{\sc Garofalo, N., and Nhieu, D.-M.}
\newblock Isoperimetric and {S}obolev inequalities for
  {C}arnot-{C}arath\'eodory spaces and the existence of minimal surfaces.
\newblock {\em Comm.\ Pure Appl.\ Math. 49}, 10 (1996), 1081--1144.

\bibitem{gp:bernstein}
{\sc Garofalo, N., and Pauls, S.}
\newblock The {B}ernstein problem in the {H}eisenberg group.
\newblock preprint, 2003.

\bibitem{GT}
{\sc Gilbarg, D., and Trudinger, N.~S.}
\newblock Elliptic partial differential equations of second order.
\newblock xiv+517.
\newblock Reprint of the 1998 edition.

\bibitem{pau:cmc-carnot}
{\sc Hladky, R.~K., and Pauls, S.~D.}
\newblock Constant mean curvature surfaces in sub-riemannian geometry.
\newblock to appear in Jour. Diff. Geom.

\bibitem{hormander}
{\sc H\"ormander, L.}
\newblock Hypoelliptic second order differential equations.
\newblock {\em Acta Math.}, 119 (1967), 147--171.

\bibitem{montefalcone}
{\sc Montefalcone, F.}
\newblock Hypersurfaces and variational formulas in sub-{R}iemannian {C}arnot
  groups.
\newblock {\em J. Math. Pures Appl. (9) 87}, 5 (2007), 453--494.

\bibitem{NSW}
{\sc Nagel, A., Stein, E.~M., and Wainger, S.}
\newblock Balls and metrics defined by vector fields. {I}. {B}asic properties.
\newblock {\em Acta Math. 155}, 1-2 (1985), 103--147.

\bibitem{pau:minimal}
{\sc Pauls, S.~D.}
\newblock Minimal surfaces in the {H}eisenberg group.
\newblock {\em Geom. Dedicata 104\/} (2004), 201--231.

\bibitem{pau:obstructions}
{\sc Pauls, S.~D.}
\newblock {$H$}-minimal graphs of low regularity in {$\Bbb H\sp 1$}.
\newblock {\em Comment. Math. Helv. 81}, 2 (2006), 337--381.

\bibitem{Rito}
{\sc Ritor{\'e}, M.}
\newblock Examples of area-minimizing surfaces in the sub-riemannian heisenberg
  group $\mathbb h^1$ with low regularity.

\bibitem{RR1}
{\sc Ritor{\'e}, M., and Rosales, C.}
\newblock Area stationary surfaces in the {H}eisenberg group {$\Bbb H\sp 1$}.
\newblock Preprint (2004).

\bibitem{RR}
{\sc Ritor{\'e}, M., and Rosales, C.}
\newblock Rotationally invariant hypersurfaces with constant mean curvature in
  the {H}eisenberg group {$\Bbb H\sp n$}.
\newblock {\em J. Geom. Anal. 16}, 4 (2006), 703--720.

\bibitem{Roth:Stein}
{\sc Rothschild, L.~P., and Stein, E.~M.}
\newblock Hypoelliptic differential operators and nilpotent groups.
\newblock {\em Acta Math. 137}, 3-4 (1976), 247--320.

\bibitem{sawhe}
{\sc Sawyer, E.~T., and Wheeden, R.~L.}
\newblock H\"older continuity of weak solutions to subelliptic equations with
  rough coefficients.
\newblock {\em Mem. Amer. Math. Soc. 180}, 847 (2006), x+157.

\bibitem{Sherbakova}
{\sc Sherbakova, N.}
\newblock Minimal surfaces in contact subriemannian manifolds,.
\newblock {\em preprint\/} (2006).

\end{thebibliography}
\end{document}